\documentclass[sn-mathphys,pdflatex]{sn-jnl}

\usepackage[utf8]{inputenc}
\usepackage{lmodern}
\usepackage{mathtools,amsfonts,amssymb,amscd,dsfont}
\usepackage[misc]{ifsym}
\usepackage[shortlabels]{enumitem}
\newlist{lista}{enumerate}{1}
\setlist[lista]{label=\alph*., nosep,leftmargin=*,align=right}

\newlist{listi}{enumerate}{1}
\setlist[listi]{label={\upshape(\roman*\upshape)},leftmargin=*,align=right, widest=iii,nosep, format=\bf}

\usepackage{etoolbox}
\makeatletter
\patchcmd{\ps@headings}%% Regular Pages Style 
{\hbox to \hsize{\hfill \mbox{ } \hfill}}
{\hbox to \hsize{\hfill \mbox{ } \hfill}}
{}
{}
\patchcmd{\ps@titlepage}%% Opening Page Style   
{\hbox to \hsize{\hfill \mbox{ } \hfill}}
{\hbox to \hsize{\hfill \mbox{ } \hfill}}
{}
{}
\makeatother

\usepackage{lmodern}
\usepackage{amssymb}
\usepackage{subcaption}
\usepackage{mathrsfs}
\usepackage[separate-uncertainty = true]{siunitx}
\usepackage{booktabs}
\usepackage{amssymb}
\usepackage{amsmath}
\DeclareMathOperator*{\argminA}{arg\,min} % Jan Hlavacek
\DeclareMathOperator*{\argmaxA}{arg\,max} % Jan Hlavacek

\DeclareMathOperator*{\dist}{dist} 
\DeclareMathOperator*{\circum}{circ}

\def\re{\mathds R}
\def\na{\mathds N}

\newcommand{\norm}[1]{\left\lVert#1\right\rVert}
\newcommand{\scal}[2]{\left\langle{#1},{#2}  \right\rangle}

\def\lV{\left\lVert }
\def\rV{\right\rVert }

\DeclareMathOperator{\Id}{Id}

\DeclareMathOperator{\inte}{int}
\DeclareMathOperator{\bound}{bd}

\DeclareMathOperator{\diag}{diag}

\newcommand{\pCRMOp}{{\mathscr{C}}}
\newcommand{\CRMOp}{{\mathcal{C}}}

\usepackage{mathrsfs}

\usepackage{xcolor}

\theoremstyle{thmstyleone}%
\newtheorem{theorem}{Theorem}[section]% meant for sectionwise numbers
\newtheorem{lemma}[theorem]{Lemma}% 
\newtheorem{corollary}[theorem]{Corollary}% 

\theoremstyle{thmstyletwo}%
\newtheorem{remark}[theorem]{Remark}%

\theoremstyle{thmstylethree}%
\newtheorem{definition}[theorem]{Definition}%

\usepackage[nameinlink]{cleveref}
\usepackage{autonum}

\Crefname{lemma}{Lemma}{Lemmas}
\crefname{proposition}{Proposition}{Propositions}
\crefname{definition}{Definition}{Definitions}
\Crefname{proposition}{Proposition}{Propositions}
\Crefname{corollary}{Corollary}{Corollaries}
\numberwithin{equation}{section}

\crefformat{equation}{\textup{#2(#1)#3}}
\crefrangeformat{equation}{\textup{#3(#1)#4--#5(#2)#6}}
\crefmultiformat{equation}{\textup{#2(#1)#3}}{ and \textup{#2(#1)#3}}
{, \textup{#2(#1)#3}}{, and \textup{#2(#1)#3}}
\crefrangemultiformat{equation}{\textup{#3(#1)#4--#5(#2)#6}}%
{ and \textup{#3(#1)#4--#5(#2)#6}}{, \textup{#3(#1)#4--#5(#2)#6}}{, and \textup{#3(#1)#4--#5(#2)#6}}

\Crefformat{equation}{Equation~\textup{(#1)}#3}
\Crefrangeformat{equation}{Equations~\textup{#3(#1)#4--#5(#2)#6}}
\Crefmultiformat{equation}{Equations~\textup{#2(#1)#3}}{ and \textup{#2(#1)#3}}
{, \textup{#2(#1)#3}}{, e \textup{#2(#1)#3}}
\Crefrangemultiformat{equation}{Equations~\textup{#3(#1)#4--#5(#2)#6}}%
{ e \textup{#3(#1)#4--#5(#2)#6}}{, \textup{#3(#1)#4--#5(#2)#6}}{, e \textup{#3(#1)#4--#5(#2)#6}}

\crefdefaultlabelformat{#2\textup{#1}#3}

\begin{document}

\title[A successive centralized circumcentered-reflection method]{A successive centralized circumcentered-reflection method for the convex feasibility problem}

%\subtitle{Using  the  LaTex Template}

\author[1,4]{\fnm{Roger} {Behling}}\email{rogerbehling@gmail.com}

\author[2]{{Yunier} \sur{Bello-Cruz}} \email{yunierbello@niu.edu}
\author[1]{\fnm{Alfredo} \sur{Iusem}} \email{alfredo.iusem@fgv.br} 
\author*[3]{\fnm{Di} \sur{Liu}}  \email{di.liu@impa.br} 
\author[4]{\fnm{Luiz-Rafael} \sur{Santos}}\email{l.r.santos@ufsc.br}

\affil[1]{\orgdiv{School of Mathematics}, \orgname{Fundação Getúlio Vargas}, \orgaddress{\city{Rio de Janeiro}-\state{RJ} -- \postcode{22250-900}, \country{Brazil}}}

\affil[2]{\orgdiv{Department of Mathematical Sciences}, \orgname{Northern Illinois University}, \orgaddress{\city{DeKalb}-\state{IL} -- \postcode{60115-2828}, \country{USA}}}

\affil[3]{\orgname{Instituto de Matematica Pura e Aplicada}, \orgaddress{\city{Rio de Janeiro}-\state{RJ} -- \postcode{22460-320}, \country{Brazil}}}

\affil[4]{\orgdiv{Department of Mathematics}, \orgname{Federal University of Santa Catarina}, \orgaddress{\city{Blumenau}-\state{SC} -- \postcode{89065-300}, \country{Brazil}}}

\date{}
%The correct dates will be entered by the editor.

\abstract{
In this paper, we present a successive centralization process for the circumcentered-reflection scheme with several control sequences for solving the convex feasibility problem in Euclidean space. Assuming that a standard error bound holds, we prove the linear convergence of the method with the most violated constraint control sequence. Moreover, under additional smoothness assumptions on the target sets, we establish the superlinear convergence. Numerical experiments confirm the efficiency of our method.}


\jyear{2023}%

\maketitle

% \tableofcontents
%All acknowledgements should be placed in the back of the paper after Conclusions..

\section{Introduction}\label{sec:intro}
The convex feasibility problem (CFP) aims to solve:
\begin{equation}
\label{eq.CFP_1}
    \text{find }x^* \in  C\coloneqq \bigcap_{i=1}^m C_i, 
\end{equation}
where each $C_i \subset \re^n$ is closed and convex, for $i=1,2,\ldots,m$. Moreover, we assume that $C\neq \emptyset$. Convex feasibility represents a modeling paradigm for solving many engineering and physics problems, \emph{e.g.}, image recovery \cite{Combettes:1996},  wireless sensor networks localization \cite{Hu:2016}, and gene regulatory network inference \cite{Wang:2017}.

Based on the orthogonal projections, a broad class of methods is available for solving problem \cref{eq.CFP_1}; see, for instance, \cite{Bauschke:1996}. Two well-known algorithms among them are the Sequential Projection Method (SePM) and the Simultaneous Projection Method (SiPM), which only use the individual projections onto $C_i$'s, $P_{C_i}$. The projection operator for each $C_i$, $P_{C_i}:\re^n\rightarrow C_i$, is given by
\begin{equation}
\label{eq.definition_Projection}
    P_{C_i}(x) =\argminA_{s \in C_i} \| x-s\|.
\end{equation}
The SePM and the SiPM operators are defined as $\Bar{P}=P_{C_m}\circ \dots \circ P_{C_1}$ and $\hat{P}=\frac{1}{m}\sum_{i=1}^m P_{C_i}$, respectively. Given $s^0, y^0\in\re^n$, we set $s^{k+1} = \Bar{P}(s^k)$ and $y^{k+1} = \hat{P}(y^k)$ with $k\in\na$ the sequences generated by SePM and SiPM, respectively. These two iterations converge to a solution of problem \cref{eq.CFP_1} if $\bigcap_{i=1}^m C_i \neq \emptyset $. Moreover, it is well-known that under some error bound conditions, to be discussed later, SePM and SiPM have linear convergence rates. Further study of these two algorithms can be found in \cite{Bauschke:1996,DePierro:1985}.

% \subsection{The circumcentered-reflection method}

In this paper we are going to use a  circumcentered-reflection scheme for solving problem \cref{eq.CFP_1}. The circumcentered-reflection  method (CRM) has been proposed in \cite{Behling:2018} to solve problem \cref{eq.CFP_1} with two closed convex sets $A,B\subset \re^n$.  CRM was first proposed to accelerate the Douglas-Rachford method (DRM)~\cite{Douglas:1956,Bauschke:2014b}, and Method of Alternating projections (MAP)~\cite{Bauschke:1993,Bauschke:2016}  (which coincides with SePM introduced above, if only two sets are considered).
During the past four years, CRM has fascinated researchers in the field of continuous optimization, resulting in an avalanche of surprising results and improvements for the circumcenter scheme;  see, for instance, \cite{Behling:2018a, Behling:2020, Behling:2021b, Behling:2021, Behling:2023, Araujo:2022, Arefidamghani:2021,Arefidamghani:2023, Bauschke:2018, Bauschke:2020, Bauschke:2021, Bauschke:2021b, Bauschke:2021d,Dizon:2022, Dizon:2022a, Lindstrom:2022,  Ouyang:2021a, Ouyang:2022b,Ouyang:2023}.

The circumcenter of three points $x,y,z\in \re^n$, noted as $\circum(x,y,z)$, is the point in $\re^n$ that lies in the affine space defined by $x,y$ and $z$ and is equidistant to these three points. CRM iterates by means of the operator $\CRMOp_{A,B}$ with respect to $A,B$ defined as 
\begin{equation}
    \label{eq.definition_CRMop}
    \CRMOp_{A,B}(x)=\circum(x,R_A(x),R_B(R_A(x)),
\end{equation}
where $R_A = 2P_A - \Id $ and $R_B = 2P_B-\Id$, and $\Id$ is the identity operator in $\re^n$.

One of the limitations of CRM is that its convergence theory requires one of the sets to be a linear manifold. A counter-example for which CRM does not converge for two general convex sets was found in \cite{AragonArtacho:2020}. It is worth noting that CRM can be used for solving the CFP with $m$ general arbitrary closed convex sets by using Pierra's product space reformulation \cite{Pierra:1984}. This method is called \emph{CRM-Prod}, which is briefly described next. Define $\mathbf{W}\coloneqq C_1 \times C_2 \times \cdots \times C_m \subset \re^{nm}$ and $\mathbf{D}\coloneqq \{(x,x,\ldots,x) \in \re^{nm}\mid x\in \re^n\}$. One can easily see that
\begin{equation}
    \label{eq.Product_Space1}
    x^* \in C \Leftrightarrow \mathbf{z}^* \coloneqq (x^*,x^*,\ldots,x^*)\in \mathbf{W}\cap \mathbf{D}.
\end{equation}
Due to \cref{eq.Product_Space1}, solving problem \cref{eq.CFP_1} corresponds to solve
\begin{equation}
    \label{eq.Product_Space2}
   \text{ find }\mathbf{z}^*\in \mathbf{W}\cap \mathbf{D}.
\end{equation}
Since $\mathbf{D}$ is an affine manifold, CRM operator \cref{eq.definition_CRMop}   can be applied to the convex sets $\mathbf{D}$ and $\mathbf{W}$ in the product space $\re^{nm}$ given rise to the CRM-Prod iteration, \emph{i.e.}, 
\[\label{eq:CRMProd}
    \mathbf{z}^{k+1} \coloneqq  \CRMOp_{\mathbf{W},\mathbf{D}}( \mathbf{z}^{k})=\circum( \mathbf{z}^{k+1},R_{\mathbf{W}}( \mathbf{z}^{k}),R_{\mathbf{D}}(R_{\mathbf{W}}( \mathbf{z}^{k})).
\]
Unfortunately, the numerical evidence in \cite{Behling:2021} showed that the cost of introducing the product space is expensive. 
% since the dimension of the problem increased fast very fast.

More recently, an extension of CRM, called the centralized circumcentered-reflection method (cCRM), was introduced in \cite{Behling:2021} for overcoming the drawback of CRM, namely the request that one of the sets be an affine manifold.
% The centralized strategy is implemented by the composition of {SePM} and {SiPM}. 
For describing cCRM, we need some notation. Suppose that $A,B \subset \re^n$ are both closed convex sets and define the \emph{alternating projection operator} for SePM $Z_{A,B}:\re^n \rightarrow \re^n$ as
\begin{equation}
\label{eq.def.2SePM}
    Z_{A,B}\coloneqq P_A\circ P_B,
\end{equation}
and the \emph{simultaneous projection operator} for SiPM $\Tilde{Z}_{A,B}:\re^n \rightarrow \re^n$ as
\begin{equation}
\label{eq.def.3SiPM}
    \Tilde{Z}_{A,B} \coloneqq \frac{1}{2} (P_A + P_B),
\end{equation}
with $P_A$, $P_B$ as in \cref{eq.definition_Projection}.
We also define operator $\Bar{Z}_{A,B}:\re^n \rightarrow \re^n$ as
\begin{equation}\label{eq.centralizationProcedure}
    \Bar{Z}_{A,B} \coloneqq \frac{1}{2}(Z_{A,B}+P_B\circ Z_{A,B})=\Tilde Z_{A,B}\circ Z_{A,B}.
\end{equation}
% \textcolor{magenta}{(This is substitute of $z_C$)}\\
Finally, instead of using sequential reflections as in \cref{eq.definition_CRMop}, cCRM relies on the \emph{parallel circumcenter operator} $\pCRMOp_{A,B}$ defined by circumcentering parallel reflections, \emph{i.e.}, for $x\in \re^n$ we have  \[
    \label{eq:paralellCRMop}\pCRMOp_{A,B}(z)\coloneqq \circum(x,R_A(x),R_B(x)).\]

With this notation, we define the cCRM operator $T_{A,B}:\re^n \rightarrow \re^n$ for sets $A$ and $B$ at $x\in\re^n$ as
\begin{align}
    T_{A,B}(x)&\coloneqq\pCRMOp_{A,B}(\Bar{Z}_{A,B}(x)) \\  &= \circum(\Bar{Z}_{A,B}(x),R_A(\Bar{Z}_{A,B}(x)),R_B(\Bar{Z}_{A,B}(x))).\label{eq.definitionOfcCRMOperator}
\end{align}
Therefore, given $x^0\in \re^n$, the cCRM method is defined by the iteration
\begin{equation}
    \label{eq.cCRM_iterationScheme}
    x^{k+1} = T_{A,B}(x^k).
\end{equation}

It has been proved in \cite{Behling:2021} that the sequence defined by cCRM converges to point $x^*\in A\cap B$ whenever $A\cap B \neq \emptyset$. Under an error bound assumption the sequence converges linearly. Moreover, under some additional smoothness hypotheses and an error bound condition, the sequence generated by cCRM was proved to have superlinear convergence \cite[Thm.~3.13]{Behling:2021}.
% With the centralized strategy, $\hat{Z}_{A,B}$ will definitely push point towards the solution set. With the {error bound} assumption, the linear convergence has been proved [{4}]. Superlinear rated also is available for cCRM with more assumptions.
%In [xx] they prove the linear (now maybe the superlinear convergence with the error bound).

In this paper, we will extend the cCRM to the case of CFP with $m$ sets. The natural way to generalize it is to choose a pair of sets among $\{C_1,C_2,\ldots,C_m\}$ at the iteration $k$ and apply cCRM to this pair of sets. We define this pair as $\ell(k)$ and $r(k)$, with $\ell(k),r(k)\in \{1,\ldots,m\}$, and then apply the operator in \cref{eq.cCRM_iterationScheme} to this pair. Therefore, the \emph{successive centralized circumcentered-reflection method} (s-cCRM) for 
 CFP with $m$ sets is defined as
% The natural way to generalize it is to use the operator defined in (1.2) to all the target sets one by one. This is also what Borwein and Tam did to Douglas Rachford method in [xx]. With the control sequence, the cCRM for the N-set CFP has the following formula
\begin{equation}
\label{eq.definitionOfCRMOperator_m_sets}
    z^{k+1} = T_{C_{r(k)},C_{\ell(k)}}(z^k),
\end{equation}
where $T_{C_{r(k)},C_{\ell(k)}}$ is the cCRM operator defined in \cref{eq.definitionOfcCRMOperator} w.r.t. $C_{\ell(k)}$ and $C_{r(k)}$. 

The sequences $\{\ell(k)\}, \{r(k)\}$, determining which sets are used at the $k$-th iteration, are called \emph{control sequences}.
In all successive projection-type methods,  control sequences considerably impact the algorithms' performance. Indeed,  strategies using  such sequences have been studied before for SePM; see, for instance, ~\cite{McCormick:1977,Censor:1981,Martinez:1985}.

{%\color{blue}
 The following control sequences, which we will use in this paper, are classical; an in-depth treatment of them can be found in \cite{Censor:1981}}. 

A natural one is the \emph{cyclic control sequence}, \emph{i.e.}, at iteration $k$ we choose
\[\label{eq:cyclic_control}
\begin{aligned}
    \ell(k)=1,2,3, \ldots, m-1, m, 1,2,\ldots,\text{  and  } \\
     r(k)=2,3,4, \ldots, m-1,m,1,2,3, \dots.\end{aligned}\] {%\color{blue}
This is the option found in the first approaches to these types of methods.
A generalization of this control sequence is the \emph{almost cyclical control sequence}, which requires that each set is used at least once in any cycle of iterations of some predetermined length.
A limitation of cyclic and almost cyclic control sequences is that they do not use} any information available at iteration $k$, \emph{e.g.}, the distances of the present iterate to the target sets.

An alternative option is the \emph{most violated constraint control sequence (distance version)}, which chooses $\ell(k)$ as the set which lies the farthest away from $z^k$, with the goal of getting closer to the intersection set $C$. The drawback is that the distance from $z^k$ to all sets $C_i$ must be calculated to determine $\ell(k)$, which in general, is computationally expensive. 

When the sets $C_i$ are represented as sublevel sets of convex functions, that is, $C_i \coloneqq \{x\in \re^n\mid f_i(x) \leq 0\}$ where $f_i:\re^n \rightarrow \re$ is convex, for all $i=1,2,\ldots,m$. In this case, which happens frequently in applications, we have the option of the \emph{most violated constraint control sequence (functional value version)}, where $\ell(k)$ is chosen so that $f_{\ell(k)}(z^k) \ge f_i(z^k)$ for all $i$, again with the expectation of getting closer to $C$.  

{
%\color{blue}

These three options (almost cyclical and most violated constraint in distance version, or functional value version)
are mutually independent. As explained above, the third one is expected to perform better than the second one, which is expected to perform better than the first one. Our numerical experiments confirm this behavior. Also, we are able to prove
linear or superlinear convergence (under adequate assumptions)
for the second and third option. We do not have results of this type for the almost cyclical control sequence.} 

Next, we formally define these three control sequences. 
\begin{definition}[Control sequences]
\label[definition]{def.ControlSequence}
We say that the sequence $\{\ell(k)\}$ is:
\begin{listi}
    \item \emph{almost cyclic}, if $1\leq \ell(k) \leq m$  and there exists an integer $Q \geq m$ such that, for all $k\geq 0$ and $\{1,2,\ldots,m\} \subset \{\ell(k+1),\ell(k+2),\ldots,\ell(k+Q)\}$. An almost cyclic control with $Q = m$ is called \emph{cyclic}; 
    % (\textcolor{magenta}{There is a problem here, how to decided$r(k)=l(k+1)$ ???})
    \item \emph{most violated constraint} (distance version),
    if 
    \begin{equation}
    \label{eq.MostViolated1}
    \begin{split}
        \ell(k) & \coloneqq \argmaxA\limits_{1\leq i \leq m} \{\dist(z^k, C_i)\} , \\
        r(k) & \coloneqq \argmaxA\limits_{1\leq i \leq m} \{\dist(P_{C_{\ell(k)}}(z^k),  {C_i}) \};
    \end{split}
    \end{equation}
    \item \emph{most violated constraint} (function value version), if we assume that the sets $C_i$ in problem \cref{eq.CFP_1} are in the form
    \begin{equation}
    \label{eq.CFP_function_values_verson}
        C_i \coloneqq \{x\in \re^n\mid f_i(x) \leq 0\},
    \end{equation}
    where $f_i:\re^n \rightarrow \re$ is convex for all $i=1,2,\ldots,m$, then the most violated constraint control sequence in the function value version is given by 
    \begin{equation}
    \label{eq.MostViolated2}
    \begin{split}
        \ell(k) & \coloneqq  \argmaxA\limits_{1\leq i \leq m} \{f_i(z^k)\}, \\
        r(k) & \coloneqq \argmaxA\limits_{1\leq i \leq m} \{f_i(P_{C_{\ell(k)}}(z^k)) \}.
    \end{split}
    \end{equation}
    % \begin{equation}
    %     \argmaxA\limits_{\theta}
    % \end{equation}
    
\end{listi}
\end{definition}

When the sets in the CFP are presented as in \cref{eq.CFP_function_values_verson}, the control sequence of the most violated constraint (function value version) depends not only on the sets themselves but also on the specific functions $f_i$ used to represent them,
which are, of course, not unique. The control sequence will change if we change the functions $f_i$ (keeping the same sets $C_i$). Here,  we assume that the $f_i$'s are fixed from the onset, so that the original problem can be seen
as that of finding a point $\bar x$ which satisfies $f_i(\bar x)\le 0$ for all $i\in \{1,2, \dots m\}$.

We note that Borwein and Tam, in \cite{Borwein:2014,Borwein:2015}, introduced and analyzed a cyclic Douglas-Rachford iteration scheme. In this paper, in addition to studying the cyclic version of cCRM, we are going to  employ the other two control sequences above. 

Let us define the terminology for addressing the different algorithmic choices of the control sequences used in s-cCRM iteration given in \cref{eq.definitionOfCRMOperator_m_sets}. \textbf{Algorithm 1} is s-cCRM with the almost cyclic control sequence presented in \Cref{def.ControlSequence}(i); \textbf{Algorithm 2} stands for s-cCRM with the most violated constraint control sequence (distance version) introduced \Cref{def.ControlSequence}(ii); and \textbf{Algorithm 3} considers the most violated constraint control sequence (function value version) set in \Cref{def.ControlSequence}(iii)  within s-cCRM.

%\subsection{Organization of the paper}
The paper is organized as follows: In \Cref{sec.preliminaries}, we give definitions and preliminaries. In \Cref{sec:convergence}, we introduce and prove the global convergence of s-cCRM, under the three control sequences defined above. In \Cref{sec:convergence_rate}, we prove linear convergence for both versions of the most violated constraint control sequence
under a standard \emph{error bound} assumption, and superlinear convergence under some additional smoothness assumptions. \Cref{sec:NumericalExperiments} presents numerical experiments comparing s-cCRM with SePM and CRM-Prod.
% have been done compared with CRM in product space and DRM.

\section{Preliminaries}\label{sec.preliminaries}
Throughout this paper, we work in $\re^n$, and the norm $\norm{\cdot }$ is the norm induced by the Euclidean scalar product $\scal{\cdot}{\cdot}$. In this section we recall several basic results needed in our convergence analysis. 

First, we introduce some orthogonal projection properties.

% Given a closed convex subset A of $\re^n$, the (nearest point) projection operator $P_A:\re^n\rightarrow A$ given by
% \begin{equation}
%     P_A(x) =\argminA_{a \in A} \| x-a\|
% \end{equation}
% The reflection operator $R_A: \re^n \rightarrow A$ given by
% \begin{equation}
%     R_A(x) = (2P_A-Id)x
% \end{equation}
% where $Id:\re^n \rightarrow \re^n$ is the identity operator.
\begin{lemma}[{Properties of projection onto convex sets}]
\label[lemma]{lem.Propery_Projection}
Let $C \subset \re^n$ be a nonempty closed convex set and $P_C$ be the orthogonal projection onto set $C$ defined in \cref{eq.definition_Projection}. Then, the following hold:
\begin{listi}
    \item For all $x,y\in \re^n$, \(\|P_C(x)-P_C(y)\|^2 \leq \|{x-y}\|^2 - \|{(P_C(x)-x) - (P_C(y)-y)}\|^2\). 
    \item For all $x \in \re^n$  and   $s\in C$, $\|P_C(x)-s\| \leq \|x-s\|$.
    \item For all $x \in \re^n$ and  $s\in C$, $\|P_C(x)-x\|^2 + \|P_C(x) - s\|^2 \leq \|x-s\|^2$. 
    \item If $C \coloneqq \bigcap_{i=1}^{m} C_i$, with $C_i, i=1\ldots,m$ being closed convex sets of $\re^n$, then  $\dist(P_{C_i}(x),C)\leq \dist(x,C)$, for all $x\in \re^n$.  
\end{listi}

\end{lemma}

\begin{proof}
    Item (i) is by \cite[Prop. 4.16]{Bauschke:2017a}, while items (ii) and (iii) are direct consequences of (i). Regarding item (iv), if $P_{C_i}(x)\in C$ we are done. Suppose it is not and let $\hat s, s \in C$ be the realizers of the distances between $P_{C_i}(x)$ and $x$ to $C$, respectively.   Then, taking into account that $s \in C_i$, we have
    \begin{align}
        \dist(P_{C_i}(x),C) &= \norm{P_{C_i}(x)-\hat s} \leq \norm{P_{C_i}(x)- s} \\        & \leq \norm{x - s} =  \dist(x,C),
    \end{align}
    where the first inequality is by the definition of distance realizers,  and the second inequality is by item (ii).
\end{proof}

 \Cref{lem.Propery_Projection}(i) means that  projections onto convex sets are \emph{firmly nonexpansive}. Note that this property is stronger than the well-known \emph{nonexpansiveness} of projections, that is, 
 \(
            \label{eq.definition_Nonepansive}
            \|P_C(x)-P_C(y)\| \leq \|x-y\|,
        \) for all $x,y\in \re^n$, with $C\subset \re^n$ being closed and convex.
        \Cref{lem.Propery_Projection}(ii) indicates that  projections are \emph{quasinonexpansive}, while \Cref{lem.Propery_Projection}(iii) says that projections are \emph{firmly quasinonexpansive}; see \cite[Def.~4.1]{Bauschke:2017a}.

We continue with several definitions and facts, beginning with the notion of \emph{Fej\'er monotonicity}.
\begin{definition}[Fej\'er monotonicity]
\label[definition]{def.FejerMonotonicity}
Suppose that $M\subset \re^n$ is nonempty. Let $\{x^k\}$ be a sequence in $\re^n$. We say $\{x^k\}$ is \emph{Fej\'er monotone} with respect to $M$ if
\(
    \| x^{k+1} - s\| \leq \|x^k-s\|,
\)
for all $s\in M$, and for all $k\in \na$.
\end{definition}
The following lemma gives the properties of Fej\'er monotone sequences.

\begin{lemma}[{Fejér monotonicity properties~\cite[Thm.~2.16]{Bauschke:1996}}]
\label[lemma]{lem.Property_FejerMOnoton_sequence}
Suppose that $M \subset \re^n$ is nonemtpy, and the sequence $\{x^k\}$ is Fej\'er monotone with respect to $M$. Then,
\begin{listi}
    \item $\{x^k\}$ is {bounded}. 
    \item For every $s\in M$, $\{\|x^k - s\|\}$ converges.
    \item If there exists a cluster point $x^*$ of $\{x^k\}$ such that $x^*\in M$, then $\{x^k\}$ converges to $x^*$.
%    \item[(iii)] $\{x^k\}$ has at most on cluster point in M.
\end{listi}
\end{lemma}

 Next, we are going to present some results about the boundedness and approximation property of convex functions.

 \begin{definition}[Local Lipschitz continuity]
 \label[definition]{def.LocallyLipschitzContinuous}
    Let $f: \re^n\rightarrow \re$ and $U\subset \re^n$. We say that $f$ is \emph{locally Lipschitz continuous} in $U$  if for every $z\in U$ there exist a constant $L>0$ and a neighborhood $V$ of $z$,  such that
     \begin{equation}
         \|f(x)-f(y)\| \leq  L \|x-y \|,
     \end{equation}
     for all $x,y \in V\cap U$.
 \end{definition}

 It is well-known that convex functions are \emph{locally Lipschitz continuous}; see  \cite[Thm.~2.1.12]{Borwein:2010}. We now recall some properties on subgradients.
 
 % \begin{lemma}
 % \label{lem.Continuity_Convex_function}
 %     Suppose that $f: \re^n \rightarrow \re$ is a convex function. Then, $f$ is \emph{locally Lipschitz continuous}.
 % \end{lemma}
 % \begin{proof}
 % See Theorem 2.1.12 in \cite{Borwein:2010}.
 % \end{proof}
 \begin{definition}[Subgradient]
     \label[definition]{def.Subgradient}
     Let $f:\re^n\rightarrow \re$ be a convex function. We say that a vector $v\in \re^n$ is a \emph{subgradient} of $f$ at a point $x\in \re^n$ if, for all $y\in \re^n$,  
     \begin{equation}
         \label{eq.Subgradient}
         f(y) \geq f(x) +\scal{v}{y-x}.
     \end{equation}

 The set of all \emph{subgradients} of a convex function at $x\in \re^n$ is called the \emph{sub\-dif\-fe\-ren\-tial} of $f$ at $x$, and is denoted by $\partial f(x)$. We present now a version of the mean value theorem for convex functions.     
\end{definition}

 \begin{lemma}[{Mean value theorem for convex functions~\cite{Wegge:1974}}]
     \label[lemma]{lem.MVT_ConvexFunction}
     Let $f:\re^n \rightarrow \re$ be a convex function and let $x$ and $y$ be vectors in $\re^n$. Then, there exists a vector $u\in \re^n$ and a subgradient $v(u)\in \partial f(u)$ such that
     \begin{equation}
         \label{eq.MVT_ConvexFunction}
         f(y) = f(x) +  \scal{v(u)}{y-x},
     \end{equation}
     where $u=\alpha x +(1-\alpha)y$, and $\alpha \in (0,1)$.
 \end{lemma}
 
 We end this section recalling that the subdifferential operator of the convex functions $f$, $\partial f:\re^n\rightrightarrows\re^n$,  is \emph{maximal monotone}~\cite[Cor.~31.5.2]{Rockafellar:1997} and \emph{locally bounded}~\cite[Thm.~3]{Qi:1983}.

 \section{Convergence analysis of s-cCRM for the multiset case}\label{sec:convergence}
We proceed to the convergence analysis of s-cCRM applied to the Convex Feasibility Problem with $m$ convex sets.

\subsection{The cCRM for two convex sets}

In this subsection we are going to present some results from  \cite{Behling:2021} regrading cCRM applied to two closed convex sets,  in order to allow us to define the  s-cCRM iteration to solve  problem \cref{eq.CFP_1}. 

First, we present the notion of \emph{centralized point}
in connection with cCRM.

\begin{definition}[Centralized point]
\label[definition]{def.CentralizedPoint}
    Let $A,B\subset \re^n$ be two nonempty closed convex sets, and a point $z\in \re^n$ is said to be \emph{centralized} with respect to $A,B$ if
    \begin{equation}
    \label{eq.CentralizedPoint}
        \langle R_A(z)-z,R_B(z)-z\rangle \leq 0.
    \end{equation}
\end{definition}

Now, we show that if we apply the operator $\Bar{Z}_{A,B}\coloneqq \Tilde{Z}_{A,B}\circ Z_{A,B}$, given in \cref{eq.centralizationProcedure}, to a point $z\in \re^n$, then the resulting point will be centralized with respect to $A$ and $B$.

\begin{lemma}[{Centralization procedure~\cite[Lem.~2.2]{Behling:2021}}]
    \label[lemma]{lem.CentralizedProcedure}
    Let $A,B\subset \re^n$ be two nonempty closed convex sets with nonempty intersection. For any $z\in \re^n$, then $\Bar{Z}_{A,B}(z)$  is centralized w.r.t. $A$ and $B$. 
\end{lemma}

We next state the firmly quasinonexpansiveness of a parallel circumcenter iteration taken from a centralized point. This result means that parallel circumcenter steps  taken from centralized points move towards the solution of the convex feasibility problem involving two intersecting sets.

\begin{lemma}[{Firmly quasinonexpansiveness of circumcenters at centralized points~\cite[Lem.~2.5]{Behling:2021}}]
\label[lemma]{lem.FirmlyNonexpansiveness_Cirmcum}
    Let $A,B\subset \re^n$ be two nonempty closed convex sets with nonempty intersection. Assume that $z\in \re^n$ is a centralized point with respect to $A,B$. Then, $\pCRMOp_{A,B}(z)$ defined  in \cref{eq:paralellCRMop} satisfies
    \begin{equation}
        \label{eq.FirmlyNonexpansiveness_Cirmcum}
        \|\pCRMOp_{A,B}(z)-s\|^2 \leq \|z-s\|^2 - \|z-\pCRMOp_{A,B}(z)\|^2,
    \end{equation}
    for all $s\in A\cap B$. 
\end{lemma}

Let us state the firmly quasinonexpansiveness of the cCRM operator. This lemma was the key result in \cite{Behling:2021} to prove the convergence of cCRM, when CFP consists of two sets.

\begin{lemma}[{Firmly quasinonexpansiveness of cCRM~\cite[Lem.~2.6]{Behling:2021}}]
    \label[lemma]{lem.Firm-Quasinonexpansiveness_cCRM}
    Let $A,B\subset \re^n$ be two nonempty closed convex sets with nonempty intersection. Let $z\in \re^n$. Then, $T_{A,B}(z)$ defined in \cref{eq.definitionOfcCRMOperator} satisfies
    \begin{equation}
        \label{eq.Firm-Quasinonexpansiveness_cCRM}
        \|T_{A,B}(z)-s\|^2 \leq \|z-s\|^2 -\frac{1}{8}\|z-T_{A,B}(z)\|^2,
    \end{equation}
    for all $s\in A\cap B$. 
\end{lemma}

\subsection{Successive cCRM for the multiset case}

We present now the results that extend the cCRM for the multiset case, that is, to solve problem \cref{eq.CFP_1}. We consider three options for the control sequence, as explained in \Cref{sec:intro}, namely almost cyclic, most violated constraint (distance version) and most violated constraint (function value version), giving rise to Algorithms 1, 2 and 3, respectively.

Initially, \Cref{lem.FirmlyNonexpansiveness_Cirmcum} is extended to the multiset case.
\begin{corollary}[Firmly quasinonexpansiveness of parallel circumcenters at centralized points]
\label[corollary]{lem.Fejermonotonicty_Circ}
Let $C_1,\ldots,C_m \subset \re^n$ be nonempty closed convex sets with nonempty intersection. Assume $i,j\in\{1,2,\ldots,m\}$ and suppose  $z\in \re^n$ is a centralized point with respect to $C_i\cap C_j$. Then, $\pCRMOp_{C_i,C_j}(z)$ defined  as in \cref{eq:paralellCRMop} satisfies
\begin{equation}
\label{eq.FejerMontone_Circ1}
    \norm{\pCRMOp_{C_i,C_j}(z) - s} \leq \|z-s\| - \|z - \pCRMOp_{C_i,C_j}(z)\|,
\end{equation}for all $s\in C\coloneqq \bigcap_{i=1}^m C_i$.
\end{corollary}
\begin{proof}
 Since $C \subset C_i\cap C_j$ for any $i,j\in\{1,2,\ldots,m\}$, the result is a direct consequence of \Cref{lem.FirmlyNonexpansiveness_Cirmcum}.\end{proof}

 In sequel, we establish that  SiPM iteration \cref{eq.def.3SiPM} is quasinonexpansive.
 
\begin{lemma}[Quasinonexpansiveness of simultaneous projections]
\label[lemma]{lem.Fejer_monotonicity_SiPM}
Let $C_1,\ldots,C_m\subset \re^n$ be nonempty closed convex sets with $C$ being their nonempty intersection. Assume that $z\in \re^n$ is an arbitrary point. Then, we have 
\begin{equation}
    \|\Tilde{Z}_{C_i,C_j}(z) - s \| \leq \|z-s\|,
\end{equation}
for each $s\in C$ and for every $i,j \in \{1,\ldots,m\}$.
\end{lemma}
\begin{proof}
    Observe that, for $z\in \re^n$ and $s\in C$, 
\begin{equation}
\begin{split}
    \|\Tilde{Z}_{C_i,C_j}(z) - s \|  & = \|\frac{1}{2}(P_{C_i}+P_{C_j})(z)-s\| \\
    & \leq \frac{1}{2}\|P_{C_i}(z) -s\| +\frac{1}{2}\|P_{C_j}(z) - s\| \\
    & \leq \|z-s\|,
\end{split}
\end{equation}
where the second inequality holds by \Cref{lem.Propery_Projection}(ii).\end{proof}

Next, the firmly quasinonexpansiveness of the s-cCRM iteration \cref{eq.definitionOfCRMOperator_m_sets} is stated as a corollary.

\begin{corollary}[Firmly quasinonexpansiveness of s-cCRM]
\label[corollary]{lem.Fejermonotonicity_cCRM_m-sets_CFP.1}
Let $C_1,C_2,\ldots,C_m$ be nonempty closed convex sets with nonempty intersection. Then, for all $z\in \re^n$, we have
\begin{equation}
\label{eq.Fejermonotonicity_cCRM_m-sets_CFP.2}
     \|T_{C_i,C_j}(z)-s\|^2 \leq \|z-s\|^2 -\frac{1}{8}\|z-T_{C_i,C_j}(z)\|^2,
\end{equation}
for all $s\in C\coloneqq  \bigcap_{i=1}^m C_i$, and for arbitrary $i,j\in \{1,\ldots,m\}$, with $i\neq j$. 

\end{corollary}
\begin{proof}
Using the fact that  $C \subset C_i\cap C_j$ for any $i,j\in\{1,\ldots,m\}$, the result is directed yielded by \Cref{lem.Firm-Quasinonexpansiveness_cCRM}.
%     By \Cref{lem.Firm-Quasinonexpansiveness_cCRM}, we have 
% \begin{equation}
% \label{eq.Fejermonotonicity_cCRM_m-sets_CFP.2}
%     \|T_{C_i,C_j}(z)-s\|^2 \leq \|z-s\|^2 -\frac{1}{8}\|z-T_{C_i,C_j}(z)\|^2,
% \end{equation}
% for each $s\in C_i\cap C_j$, and for $i,j\in\{1,2,\ldots,m\}$. Hence,
% \begin{equation}
% \label{eq.Fejermonotonicity_cCRM_m-sets_CFP.3}
%     \|T_{C_i,C_j}(z) -s \| \leq \|z -s\|,
% \end{equation}
% for each $s\in \bigcap_{i=1}^m C_i \subset C_i\cap C_j$. 

% \textcolor{red}{ \bf .... This proof can be removed!!! This can be a direct Corollary ....}

\end{proof}
We state the Fejér monotonicity of s-cCRM, which  follows from \cref{eq.Fejermonotonicity_cCRM_m-sets_CFP.2} immediately.

\begin{corollary}[Fejér monotonicity of s-cCRM]
    \label[corollary]{cor.Fejermonotonicity_cCRM_m-sets_CFP.1}
    Let $C_1,\ldots,C_m\subset \re^n$ be  nonempty closed convex sets with nonempty intersection $C$, and suppose that the sequence $\{z^k\}\subset \re^n$ is generated by s-cCRM defined in \cref{eq.definitionOfCRMOperator_m_sets} with any of the  control sequences given in \Cref{def.ControlSequence}. Then $\{z^k\}$ is Fej\'er monotone with respect to $C$.  
\end{corollary}
\begin{proof}
    According to \cref{eq.Fejermonotonicity_cCRM_m-sets_CFP.2}, we have
    \begin{equation}
        \label{}
       \norm{z^{k+1} - s } = \|T_{C_i,C_j}(z^k)-s\|^2 \leq \|z^k-s\|^2,
    \end{equation}
    for all $s\in C$ and for arbitrary $i,j\in \{1,\ldots,m\}$, with $i\neq j$. Therefore, the sequence $\{z^k\}$ is Fej\'er monotone with respect to $C$. 
\end{proof}

Now, the asymptotic convergence of s-cCRM is stated and proved. 

\begin{corollary}[Asymptotic convergence of s-cCRM]
    \label[corollary]{cor.Asyptotically_cCRM}
    Let $C_1,\ldots,C_m\subset \re^n$ be  nonempty closed convex sets, and assume $C\coloneqq \bigcap_{i=1}^m C_i\neq \emptyset$. Suppose that the sequence $\{z^k\}\subset \re^n$ is generated by s-cCRM defined in \cref{eq.definitionOfCRMOperator_m_sets} with any of the control sequences given in \Cref{def.ControlSequence}. Then,
    \begin{equation}
    \label{eq.Asyptotically_cCRM.1}
       \|z^{k+1} - z^k\| \to 0.
    \end{equation}

\end{corollary}
\begin{proof}
    By using \Cref{lem.Fejermonotonicity_cCRM_m-sets_CFP.1}, with $z = z^k$, we get
    \begin{equation}
        \label{eq.Asyptotically_cCRM.2}
        \frac{1}{8}\| T_{C_{r(k)},C_{\ell(k)}}(z^k) -z^k \|^2 \leq \|z^k-s\|^2 - \|T_{C_{r(k)},C_{\ell(k)}}(z^k)-s\|^2,
    \end{equation}
    for any $s\in C$. The definition of s-cCRM for CFP with $m$ sets gives us
    \begin{equation}
        \label{eq.Asyptotically_cCRM.3}
        \frac{1}{8} \|z^{k+1} -z^k \|^2 \leq \|z^k-s \|^2 -\|z^{k+1}-s\|^2.
    \end{equation}
    From \Cref{cor.Fejermonotonicity_cCRM_m-sets_CFP.1}, the sequence $\{z^k\}$ is Fejér monotone with respect to $C$. Hence,  \Cref{lem.Property_FejerMOnoton_sequence}(ii) establishes the result.\end{proof}
The following lemma is keystone for  derive convergence of our proposed algorithms.
\begin{lemma}
\label[lemma]{lem.SupplyToConvergenceAnalysis} Let $C_1,\ldots,C_m\subset \re^n$ be  nonempty closed convex sets, and assume $C\coloneqq \bigcap_{i=1}^m C_i\neq \emptyset$.  
Suppose that the sequence $\{z^k\}$ is generated by s-cCRM, as defined in \cref{eq.definitionOfCRMOperator_m_sets} with any of the control sequences given in \Cref{def.ControlSequence}. Then, we have
\begin{equation}
\label{eq.SupplyToConvergenceAnalysis.1}
    \|z^{k+1}-s\| \leq \|P_{C_{\ell(k)}}(z^k)-s\|,
\end{equation}\
for all $s\in C$. 
\end{lemma}
\begin{proof}
    Note that, for $s\in C$, 
\begin{align}
    \|z^{k+1} -s \| & = \|T_{C_{r(k)},C_{\ell(k)}}(z^k)-s\| 
     \leq \|\Bar{Z}_{C_{r(k)},C_{\ell(k)}}(z^k) - s\| \\
    & \leq \|Z_{C_{r(k)},C_{\ell(k)}}(z^k) -s \| =  \| P_{C_{r(k)}} (P_{C_{\ell(k)}}(z^k)) -s \|
    \\ & \leq \| P_{C_{\ell(k)}}(z^k) -s \|,\label{eq.SupplyToConvergenceAnalysis.2}
\end{align}
where the first inequality follows from  \Cref{lem.CentralizedProcedure} and \Cref{lem.Fejermonotonicty_Circ}, the second inequality holds by \Cref{lem.Fejer_monotonicity_SiPM}, and the last one is due to \Cref{lem.Propery_Projection}(ii).\end{proof}

% \section{Convergence Analysis}
% {As the cCRM is kind of composition of SiPM and other algorithms we can use the idea of the proof for SiPM to prove the convergence. The following lemma shows the Fej\'er-monotonicity of the sequence generated by composition cCRM is Fej\'er monotone}

According to the last results, we only need  to prove that there exists a cluster point of the s-cCRM sequence lying in the intersection of the underlying sets, in order to achieve the convergence of the proposed method. We establish this in sequel, for  Algorithms 1, 2 and 3. 
%-------The first result is about the almost cyclic control sequence-----

\begin{theorem}[Convergence of Algorithm 1]
\label{thm.Convergence_AlmostCyclic1} Let $C_1,\ldots,C_m\subset \re^n$ be  nonempty closed convex sets, and assume $C\coloneqq \bigcap_{i=1}^m C_i\neq \emptyset$. 
Suppose that $z^0$ is an arbitrary point in $\re^n$ and the sequence $\{z^k\}$ is generated by s-cCRM operator defined in \cref{eq.definitionOfCRMOperator_m_sets} with almost cyclic control sequence. Then, there exists some point $x^* \in C\coloneqq \bigcap_{i=1}^m C_i$ such that $z^k \rightarrow x^*$. 
\end{theorem}
\begin{proof}
    By \Cref{cor.Fejermonotonicity_cCRM_m-sets_CFP.1} and \Cref{lem.Property_FejerMOnoton_sequence}(ii), we get that $\{z^k\}$ is a Fej\'er monotone sequence w.r.t. $C$, so it is bounded. Hence, there exist a subsequence $\{z^{k_j}\}$ of $\{z^k\}$ and a point $x^* \in \re^n$ such that $z^{k_j} \rightarrow x^*$. By the definition of {almost cyclic control sequence}, for each $q\in\{1,\ldots,m\}$, there exists a sequence $\{h_j\}$ which satisfies $k_j \leq h_j \leq k_j+Q$ such that $\ell(h_j) =q $, for each $j\geq 1$. By the triangle inequality,
    \begin{equation}
        \label{eq.Convergence_AlmostCyclic1.1}
        \|z^{h_j} -z^{k_j}\| \leq \sum_{i=0}^{h_j-k_j-1} \|z^{k_j+i+1} -z^{k_j+i}\| \leq \sum_{i=0}^{Q-1} \|z^{k_j+i+1} -z^{k_j+i}\|.
    \end{equation}
    Note that the rightmost side of \cref{eq.Convergence_AlmostCyclic1.1} is a sum of $Q$ terms, and each one of them converges to $0$ as $j\rightarrow \infty$ by \Cref{cor.Asyptotically_cCRM}. Hence, the whole summation goes to $0$, so that
    \begin{equation}
        \label{eq.Convergence_AlmostCyclic1.2}
        \lim_{j\rightarrow \infty} \lVert z^{h_j} - z^{k_j}\rVert = 0.
    \end{equation}
    Since $z^{h_j} = z^{k_j} + (z^{h_j} - z^{k_j})$ and $z^{k_j} \rightarrow x^*$, by assumption, we conclude from \cref{eq.Convergence_AlmostCyclic1.2} that $\lim_{j\rightarrow \infty} z^{h_j} -x^* = 0$. Note that
    \begin{align}
            \|z^{h_j+1}-s \|^2 & \leq \| P_{C_{\ell(h_j)}}(z^{h_j})-s \|^2 \\
            & \leq \|z^{h_j}-s\|^2 - \|P_{C_{\ell(h_j)}}(z^{h_j}) - z^{h_j} \|^2,    \label{eq.Convergence_AlmostCyclic1.3}
    \end{align}
    using \Cref{lem.SupplyToConvergenceAnalysis} in the first inequality, and \Cref{lem.Propery_Projection}(iii) in the second one. Therefore, 
    \begin{equation}
    \label{eq.Convergence_AlmostCyclic1.4}
        \|P_{C_{\ell(h_j)}}(z^{h_j}) - z^{h_j}\|^2 \leq \|z^{h_j} - s\|^2 - \|z^{h_j+1} - s\|^2.
    \end{equation}

    Since for all $\ell(h_j)=q$, in view of \Cref{lem.Property_FejerMOnoton_sequence}(ii), taking limit with $j\rightarrow \infty $ in \cref{eq.Convergence_AlmostCyclic1.4}, we get that the right side of \cref{eq.Convergence_AlmostCyclic1.4} goes to $0$. Consequently, 
    \begin{equation}
        \|P_{C_q}(x^*) - x^*\|^2 \leq  0.
    \end{equation}
    Hence, we obtain that $P_{C_q}(x^*)=x^*$. Since $q$ is an arbitrary index, we have that $x^* \in C_{q}$ for all $q\in\{1,2,\ldots,m\}$. Consequently, $x^*\in C$. By \Cref{lem.Property_FejerMOnoton_sequence}(iii), we get that $z^k \rightarrow x^* \in C$.
    
\end{proof}

\begin{remark} We know that the {cyclic} control sequence is a special case of the almost cyclic control sequence, so we conclude that the sequence $\{z^k\}$ generated by s-cCRM defined in \cref{eq.definitionOfCRMOperator_m_sets} will converge to some point in $C$ if we use a cyclic control sequence.
\end{remark}
%-------------end here-------------------------------------

%------------convergence for the Most Violatedconstraints-------------
\begin{theorem}[Convergence of Algorithm 2]
\label{thm.Convergence_MostViolated1}
Suppose that $z^0$ is an arbitrary point in $\re^n$ and the sequence $\{z^k\}$ is generated by the s-cCRM defined in \cref{eq.definitionOfCRMOperator_m_sets} with the most violated constraint control sequence in distance version \cref{eq.MostViolated1}. Then there exists some point $x^* \in C \coloneqq \bigcap_{i=1}^m C_i$ such that $z^k \rightarrow x^*$.
\end{theorem}
\begin{proof}
    Following the same lines of the proof of last theorem,  \Cref{cor.Fejermonotonicity_cCRM_m-sets_CFP.1} and \Cref{lem.Property_FejerMOnoton_sequence}(ii), imply that there exists a subsequence $\{z^{k_j}\}$ of ${z^k}$, and a point $x^* \in \re^n$ such that $z^{k_j} \rightarrow x^*$. Take any $i\in\{1,\ldots,m\}$, then
\begin{equation}
\label{eq.Convergence_MostViolated1.1}
    \begin{split}
        {\dist}^2(z^{k_j},C_i) & \leq {\dist}^2(z^{k_j},C_{\ell(k_j)}) = \|z^{k_j} - P_{C_{\ell(k_j)}}(z^{k_j})\|^2 \\
        & \leq \|z^{k_j} - s \|^2 - \|P_{C_{\ell(k_j)}}(z^{k_j}) - s \|^2 \\
%        & \leq \|z^{k_j} - s \|^2 - \|P_{r(k_j)}(P_{\ell(k_j)}(z^{k_j})) - s \|^2 \\
        & \leq \|z^{k_j} - s \|^2 - \|z^{k_j+1} - s\|^2,
    \end{split}
\end{equation}
using the definition of the most violated constraint control sequence in the first inequality, \Cref{lem.Propery_Projection}(ii) in the second one and \Cref{lem.SupplyToConvergenceAnalysis} in the third one. Take $j \rightarrow \infty$, and use \Cref{lem.Property_FejerMOnoton_sequence}(ii) for proving that the rightmost expression in \cref{eq.Convergence_MostViolated1.1} converges to $0$. Hence,
\begin{equation}
\label{eq.Convergence_MostViolated1.2}
    \dist(x^*,C_i) \leq 0,~\forall i\in \{1,2,\ldots,m\}.
\end{equation}
Since $i$ is an arbitrary index, we conclude that $x^* \in C_i$ for all $i\in \{1,\ldots,m\}$, so that $x^* \in C$. The result follows immediately by \Cref{lem.Property_FejerMOnoton_sequence}(iii).\end{proof}
%-------------------end here-------------------------

%-------------------convergence for the most violatedconstrains in Function value------
\begin{theorem}[Convergence of Algorithm 3]
\label{thm.Convergence_MostViolated2}
Consider problem \cref{eq.CFP_1}, suppose $z^0$ is an arbitrary point in $\re^n$ and assume that sequence $\{z^k\}$ is generated by the s-cCRM defined in \cref{eq.definitionOfCRMOperator_m_sets} with the most violated constraint control sequence in function value version \cref{eq.MostViolated2}. Then there exists some point $x^* \in  C\coloneqq \bigcap_{i=1}^m C_i$ such that $z^k \to x^*$. 
\end{theorem}
\begin{proof}
    Again, invoking  \Cref{cor.Fejermonotonicity_cCRM_m-sets_CFP.1} and \Cref{lem.Property_FejerMOnoton_sequence}(ii), we get a subsequence $\{z^{k_j}\}$  of ${z^k}$ and a point $x^* \in \re^n$ such that $z^{k_j} \rightarrow x^*$. Using the locally Lipschitz continuity of the convex functions $f_i$'s, for  $i\in \{1,\ldots,m\}$, there exists a neighborhood $V_i$ of $x^*$ such that $f_i$ is Lipschitz continuous in $V_i$ with constant $L_i$. Take $V=\bigcap_{i=1}^m V_i$ and $L=\max_{1\leq i \leq m} L_i$, so $f_i$ is Lipschitz continuous in $V$ with constant $L$ for each $i=1,2,\ldots,m$.
    
    For large enough $k_j$, we have 
\begin{equation}
\label{eq.Convergence_MostViolated2.1}
    \begin{split}
        f_i(z^{k_j}) & \leq f_{\ell(k_j)}(z^{k_j}) \leq f_{\ell(k_j)}(z^{k_j}) - f_{\ell(k_j)}(P_{C_{\ell(k_j)}}(z^{k_j}))  \\
        & \leq L \|z^{k_j} - P_{C_{\ell(k_j)}}(z^{k_j})\| \\
        & \leq L (\|z^{k_j}-s\|^2 - \|P_{C_{\ell(k_j)}}(z^{k_j})-s\|^2)^{\frac{1}{2}} \\
        & \leq L (\|z^{k_j}-s\|^2 - \|z^{k_j+1}-s\|^2)^{\frac{1}{2}},
    \end{split}
\end{equation}
using the definition of the most violated control sequence in the first inequality, the definition of orthogonal projection and \cref{eq.CFP_function_values_verson} in the second, the Lipschitz continuity in the third one, and \Cref{lem.Propery_Projection}(iii) and \cref{eq.SupplyToConvergenceAnalysis.1} in the fourth one. Taking $j \rightarrow \infty$, we have from \Cref{lem.Property_FejerMOnoton_sequence}(ii) that the rightmost expression in \cref{eq.Convergence_MostViolated2.1} converges to $0$. Hence, 
\begin{equation}
\label{eq.Convergence_MostViolated2.2}
    f_i(x^*) \leq 0.
\end{equation}
Since $i$ is an arbitrary index, we get that $f_i(x^*) \leq 0$ for all $i\in\{1,\ldots,m\}$, and hence $x^*\in C$. In view of \Cref{lem.Property_FejerMOnoton_sequence}(iii) and \Cref{cor.Fejermonotonicity_cCRM_m-sets_CFP.1}, $z^k \rightarrow x^* \in C$\end{proof}
%----------------------end-------------------
\section{Linear and superlinear convergence rate}\label{sec:convergence_rate}

In this section, we first introduce two options of error bounds for CFP. Then, we review the proofs of the linear convergence of SiPM and SePM under these error bounds. Next, we prove, also under these an error bound, the linear convergence of s-cCRM when using most violated control sequences, \emph{i.e.},  Algorithms 2 and 3. Finally, we prove that under a smoothness assumption on the sets $C_i$ and a Slater condition, Algorithms 2 and 3 versions of s-cCRM achieve a superlinear convergence rate. These results are in agreement with those established in \cite{Behling:2021} for cCRM applied to CFP with two sets. 

\subsection{Error bounds for CFP}

Error bound conditions are regularity assumptions under which convergence rates of projection-type schemes to solve problem \cref{eq.CFP_1} have been studied; see, for instance,~\cite{Beck:2003,Bauschke:1993,Bauschke:1999,Behling:2021a,Liu:2022}. We start with \Cref{def.EB_2setsCFP} below, regarding the error bound when two closed convex sets are considered. This error bound is also called \emph{local (Lipschitz) linear regularity}~\cite[Def.~3.11]{Bauschke:1993}, which in turn can be seen as  \emph{subtransversality}~\cite[Thm.~1]{Kruger:2018a}.

\begin{definition}[Error bound for two sets]
\label[definition]{def.EB_2setsCFP}
Let $A, B \subset \re^n$ be closed  convex and assume that $A\cap B \neq \emptyset$. We say that $A$ and $B$ \emph{satisfy a local error bound condition} if for some point $\Bar{z} \in X\cap Y$, there exist a real number $\omega \in (0,1)$, and a neighborhood $V$ of $\Bar{z}$ such that
\begin{equation}
\label{eq.EB_2setsCFP}
    \omega \dist(z,A\cap B) \leq \max \{\dist(z,A),\dist(z,B)\},
\end{equation}
for all $z \in V$.
\end{definition}
Under this condition, a point in  $V$ cannot be too close to both $A$ and $B$, and  at the same time, far from $A \cap B$. This assumption was used in \cite{Behling:2018a} and \cite{Behling:2021} to prove the linear convergence rate for CRM and cCRM. Now we extend \Cref{def.EB_2setsCFP} to the multi-set case.

\begin{definition}[EB 1]
\label[definition]{def.EB1}
Let $C_1,C_2,\ldots,C_m \subset \re^n$ be nonempty closed  convex sets, and assume that $C\coloneq \bigcap_{i=1}^m C_i \neq \emptyset$. We say that $C_1,C_2,\ldots,C_m$ satisfy \emph{the local error bound condition (EB1) at a point $\Bar{z} \in C$}, if there exists a real number $\omega \in (0,1)$, and a neighborhood $V$ of $\Bar{z}$ such that
\begin{equation}
\label{eq.EB1}
    \omega \dist(z,C) \leq \max_{1\leq i \leq m} \dist(z,C_i),\tag{EB1}
\end{equation}
for all $z\in V$.
\end{definition}

If the underlying convex sets are defined by means of convex inequalities,  we may consider the following error bound condition.

% Next Error Bound is related to the function value.
% If the CFP is in the form of (1.3) then we have the following error bound 2.

\begin{definition}[EB 2]
\label[definition]{def.EB2}
Let $C_1,C_2,\ldots,C_m \subset \re^n$ be nonempty closed convex sets, and assume that $C\coloneq \bigcap_{i=1}^m C_i\neq \emptyset$, with     
\(
C_i \coloneqq \{x\in \re^n\mid f_i(x) \leq 0\},
\)
and $f_i:\re^n \rightarrow \re$ being convex, for all $i=1,2,\ldots,m$. We say that $C_1,C_2,\ldots,C_m $ satisfy  \emph{the local error bound condition 2 (EB2) at a point $\Bar{z} \in C$}, if there exists a real number $\omega \in (0,1)$, and a neighborhood $V$ of $\Bar{z}$ such that
\begin{equation}
\label{eq.EB2}
    \omega \dist(z,C) \leq \max_{1\leq i \leq m} f_i(z),  \tag{EB2}
\end{equation}
for all $z\in V\setminus C$.
\end{definition}

\begin{remark} Using the definition of the most violated constraint control sequence \cref{eq.MostViolated2}, equation \cref{eq.EB2} becomes, for all $z\in V$, 
\begin{equation}
    \label{eq.EB2together_with_control_sequence}
    \omega \dist(z,C) \leq f_{\ell}(z),
\end{equation}
where $\ell \coloneqq \argmaxA\limits_{1\leq i \leq m} \{f_i(z)\}$.

\end{remark}

Note that  \cref{eq.EB1} and  \cref{eq.EB2} are clearly connected. Indeed, if \cref{eq.EB1} holds, we can get  \cref{eq.EB2} satisfied by setting $f_i(z) \coloneqq \dist(z,C_i)$, for $i=1,\ldots, m$. Now, since convex functions are locally Lipschitz continuous, we get that  \cref{eq.EB2} always implies  \cref{eq.EB1}. 

% \begin{lemma}[EB with the {Most violated Constraint distance}]
% Suppose that
% \begin{equation}
%     \ell(k) =arg\max_{1 \leq j \leq m} \{\dist(x),C_j\}
% \end{equation}
% and
% \begin{equation}
%     r(k) = arg\max_{1 \leq j \leq m} \{\dist (P_{C_{\ell(k)}}(z)),C_j\}
% \end{equation}
% Then the EBs become
% \begin{equation}
%     \omega \dist(z,C) \leq \dist (z,C_{\ell(k)})
% \end{equation}
% and 
% \begin{equation}
%     \omega \dist(P_{C_{\ell(k)}}(z),C) \leq \dist(P_{C_{\ell(k)}}(z), C_{r(k)})
% \end{equation}
% \end{lemma}

\begin{definition}[Convergence rate]
\label[definition]{def.LinearConvergence}
Let $\{x^k\} \subset \re^n$ be a sequence converging to some point $\Bar{x} \in \re^n$. Assume that $x^k \neq \Bar{x}$ for all $k\in \mathds{N}$. Define
\begin{equation}
\label{eq.LinearConvergence}
    \xi\coloneqq \limsup_{k \rightarrow \infty}  \frac{\|x^{k+1}-\Bar{x}\|}{\|x^k-\Bar{x}\|},\text{ and }\rho \coloneqq  \limsup_{k \rightarrow \infty}  \| x^k -\Bar{x}\|^{\frac{1}{k}}.
\end{equation}
Then, the convergence of $\{x^k\}$ is
\begin{listi}
    \item \emph{Q-linear} if $\xi\in (0,1)$;   
     \item \emph{Q-superlinear} if $\xi=0$,
    \item \emph{R-linear} if $\rho\in (0,1)$.
\end{listi}
\end{definition}

It is long-familiar that Q-linear convergence is a sufficient condition for R-linear convergence (with the same asymptotic constant), but the
converse statement does not hold true~\cite{Ortega:2000}.

%-------------------------------------------------------------------------
% \subsection{Properties about the convergence rate of cCRM}

The next lemma claims the linear convergence of Fej\'er monotone sequences.

\begin{lemma}[{Fejér monotonicity and linear convergence \cite[Prop.~3.7]{Behling:2021}}]
\label[lemma]{lem.ConvergenceRate_Fejermonoton}
If the sequence $\{x^k\}\subset \re^n$ is Fej\'er Monotone with respect to a set $M\subset \re^n$, and the scalar sequence $\{\dist(x^k,M)\}$ converges {Q-linearly} to 0, then $\{x^k\}$ converges {R-linearly} to a point $\Bar{x}\in M$. 
\end{lemma}
% \begin{proof}
% See Proposition 3.8 in \cite{Behling:2021}    
% \end{proof}

%------------------------------------------------
\subsection{Linear convergence rate of Algorithm 2}
% In this section, we will prove the linear rate for s-cCRM with EB1 first. Then prove the linear rate under EB2.
% \subsubsection{Linear convergence for SiPM and SPM under EB}

Before proving the linear convergence rate of Algorithm 2, we  recall the proof of the linear convergence for SePM and {SiPM} with the most violated constraint control related to \cref{eq.MostViolated1}. We note that, taking into account \cref{def.EB1,def.EB2}, we assume   $\omega \in (0,1)$. 

\begin{lemma}[Linear convergence of SePM and SiPM under EB1]
\label[lemma]{lem.LinearConvergence_SiPM_SePM_EB1} 
Let $C_1,\ldots,C_m\subset \re^n$ be  nonempty closed convex sets, and suppose $C\coloneqq \bigcap_{i=1}^m C_i \neq \emptyset$
Assume that \cref{eq.EB1} holds at $\Bar{z} \in C $. 
Let $B$ be a ball centered at $\Bar{z}$ and contained in $V$. Let $z\in B$ and define $\beta \coloneqq \sqrt{1-\omega^2}$. Then, 
\begin{equation}
    \label{eq.LinearConvergence_SiPM_SePM_EB1.1}
        \dist(Z_{C_{r},C_{\ell}}(z),C) \leq \beta^2 \dist(z,C),
    \end{equation}
and 
    \begin{equation}
    \label{eq.LinearConvergence_SiPM_SePM_EB1.2}
        \dist(\Tilde{Z}_{C_{r},C_{\ell}}(z),C) \leq \left(\frac{1+\beta}{2}\right)\dist(z,C),
    \end{equation}
where $\ell  \coloneqq \argmaxA\limits_{1\leq i \leq m} \{\dist(z, C_i)\} $ and  $ r \coloneqq \argmaxA\limits_{1\leq i \leq m} \{\dist(P_{C_{\ell}}(z),  {C_i}) \}$. 
 
\end{lemma}
\begin{proof}
    We start with the SePM case. Note that
        \begin{align}
            {\dist}^2(z,C) & = \|z-P_C(z)\|^2  \geq \| P_{C_{\ell}}(z) - P_C(z)\|^2 +\|z-P_{C_{\ell}}(z) \|^2 \\
            & \geq {\dist}^2(P_{C_{\ell}}(z),C) + {\dist}^2(z,C_{\ell}) \\
            & = {\dist}^2(P_{C_{\ell}}(z),C) + \max_{1\leq i \leq m}{\dist}^2(z,C_{i}) \\
            & \geq {\dist}^2(P_{C_{\ell}}(z),C) +\omega^2 {\dist}^2(z,C),
        \end{align}
    where the first inequality holds by \Cref{lem.Propery_Projection}(iii), in the second one we use the definition of $P_{C_{\ell}}$, the last equality follows from the definition of $\ell$, and the third inequality follows by  \cref{eq.EB1}. Hence,
    \begin{equation}
    \label{eq.LinearConvergence_SiPM_SePM_EB1.4}
        \dist(P_{C_{\ell}}(z),C) \leq \sqrt{1-\omega^2} \dist(z,C) =\beta \dist(z,C).
    \end{equation}
    Since $\Bar{z} \in C$, by the nonexpansiveness of $P_{C_{\ell}}$ and $P_{C_{r}}$, we have
    \begin{equation}
    \label{eq.LinearConvergence_SiPM_SePM_EB1.5}
        \|P_{C_{r}}(P_{C_{\ell}}(z)) - \Bar{z}\| \leq \|P_{C_{\ell}}(z) -\Bar{z} \| \leq \|z-\Bar{z}\|,
    \end{equation}
    hence, $P_{C_{r}}(P_{C_{\ell}}(z)) \in B$. Consequently,
  
    \begin{align}
        {\dist}^2(P_{C_{\ell}}(z),C) & = \|P_{C_{\ell}}(z)-P_C(P_{C_{\ell}}(z)) \|^2 \\
        & \geq \|P_{C_{r}}(P_{C_{\ell}}(z)) - P_C(P_{C_{\ell}}(z))\|^2  + \|P_{C_{\ell}}(z) - P_{C_{r}}(P_{C_{\ell}}(z))\|^2 \\
         & \geq {\dist}^2(P_{C_{r}}(P_{C_{\ell}}(z)),C)   +\max_{1\leq i \leq m} {\dist}^2(P_{C_{i}}(z),C_{r}) \\
        & \geq {\dist}^2(P_{C_{r}}(P_{C_{\ell}}(z)),C) + \omega^2 {\dist}^2(P_{C_{\ell}}(z),C),  \label{eq.LinearConvergence_SiPM_SePM_EB1.6}
    \end{align}
    where the first inequality follows from \Cref{lem.Propery_Projection}(iii), the second follows from the definition of $r$,    and the third one from \cref{eq.EB1}. From \cref{eq.LinearConvergence_SiPM_SePM_EB1.6}, we obtain 
    \begin{equation}
    \label{eq.LinearConvergence_SiPM_SePM_EB1.7}
        \dist(P_{C_{r}}(P_{C_{\ell}}(z)),C) \leq \sqrt{1-\omega^2} \dist (P_{C_{\ell}}(z),C) = \beta \dist (P_{C_{\ell}}(z),C).
    \end{equation}
    Now combining \cref{eq.LinearConvergence_SiPM_SePM_EB1.5} and \cref{eq.LinearConvergence_SiPM_SePM_EB1.7}, we get
    \begin{equation}
    \label{eq.LinearConvergence_SiPM_SePM_EB1.8}
    \begin{split}
        \dist(Z_{C_{r},C_{\ell}}(z),C) & =  \dist(P_{C_{r}}(P_{C_{\ell}}(z)),C)  \leq  \sqrt{1-\omega^2} \dist (P_{C_{\ell}}(z),C) \\
        & \leq (1-\omega^2) \dist (z,C) =\beta^2 \dist(z,C),  
    \end{split}
    \end{equation}
    which establishes \cref{eq.LinearConvergence_SiPM_SePM_EB1.1}.
    % where $z_{MAP}$ is the alternating projection w.r.t sets $C_{\ell}$ and $C_{r}$ which are the first and second farest sets. ({related with the control sequence})

    Next, we will establish the linear convergence for SiPM with \cref{eq.EB1}. By the nonexpansiveness of $P_{C_{r}}$, we have that
    % \begin{equation}
    %     \dist(P_{C_{\ell}}(z),C) \leq \beta \dist(z,C)
    % \end{equation}
    \begin{equation}
    \label{eq.LinearConvergence_SiPM_SePM_EB1.9}
        \dist(P_{C_{r}}(z),C) \leq \dist (z, C).
    \end{equation}
    Note that
    \begin{equation}
    \label{eq.LinearConvergence_SiPM_SePM_EB1.10}
        \begin{split}
            \dist(\Tilde{Z}_{C_{r},C_{\ell}}(z),C) & =\dist \left(\frac{1}{2}[P_{C_{\ell}}(z)+P_{C_{r}}(z)],C\right) \\
            & \leq \frac{1}{2} \left [ \dist (P_{C_{\ell}}(z),C) +\dist(P_{C_{r}}(z),C) \right ] \\
            & \leq \left( \frac{1+\beta}{2} \right) \dist(z,C), \\  
        \end{split}
    \end{equation}
    using the convexity of the distance function to $C$ in the first inequality, and \cref{eq.LinearConvergence_SiPM_SePM_EB1.4,eq.LinearConvergence_SiPM_SePM_EB1.9} in the second one. Hence, we get \cref{eq.LinearConvergence_SiPM_SePM_EB1.2}, and the results hold.\end{proof}

\begin{corollary}[Linear convergence of SePM and SiPM under EB1]
\label[corollary]{cor.LineaConvergence_SePM_SiPM_CFP_mSets}
Let $C_1,C_2,\ldots,C_m\subset \re^n$ be nonempty closed convex sets, and assume that $C\coloneqq \bigcap_{i=1}^m C_i\neq \emptyset$. Suppose $\{s^k\}$ and $\{y^k\}$ are  sequences generated by SePM and SiPM starting from some $s^0 \in \re^n$ and $y^0\in \re^n$, respectively. Assume also $\{s^k\}$ and $\{y^k\}$ are both infinite sequences. If \cref{eq.EB1} holds at the limit points $\Bar{s}$ of $\{s^k\}$, $\Bar{y}$ of $\{y^k\}$, then the sequences $\{s^k\}$, $\{y^k\}$ converge R-linearly, with asymptotic constants bounded above by $\beta^2$, $\frac{1+\beta}{2}$ respectively, where $\beta=\sqrt{1-\omega^2}$ and $\omega$ is the constant in \cref{def.EB1}. 
\end{corollary}
\begin{proof}
    Convergence of $\{s^k\}$ and $\{y^k\}$ to points $\Bar{s} \in C$ and $\Bar{y} \in C$ follows from \cite[Corollary 3.3(i)]{Bauschke:1996} and  in \cite[Theorem 3]{DePierro:1985}, respectively. Hence, for large enough $k$, $s^k$ belongs to a ball centered at $\Bar{s}$ contained in $V$, and $y^k$ belongs to a ball centered at $\Bar{y}$ contained in $V$.

    In view of the definitions of the SePM and SiPM sequences, we get from \Cref{lem.LinearConvergence_SiPM_SePM_EB1}, 
    \begin{equation}
        \frac{\dist(s^{k+1},C)}{\dist(s^k,C)} \leq \beta^2,~~~\frac{\dist(y^{k+1},C)}{\dist(y^k,C)} \leq \frac{1+\beta}{2}.
    \end{equation}
    Now, from \Cref{def.LinearConvergence}, we get 
    \begin{equation}
        \dist(s^k,C) \leq \left(\frac{1+\beta}{2}\right)^{k} \dist(s^0,C), \quad \text{ and} \quad  \dist(y^k,C) \leq \beta^{2k} \dist(y^0,C).
    \end{equation}
 Hence, both distance sequences $\{\dist(s^k, C)\}$ and $\{\dist(y^k,C)\}$ converge {Q-linearly} to $0$, with asymptotic constants given by $\beta^2$,  $\frac{1+\beta}{2}$, respectively, since $\beta\in (0,1)$, because $\omega\in(0,1)$.

    The fact that $\{s^k\}$ and $\{y^k\}$ are Fej\'er monotone with respect to $C$ is an immediate consequence of \Cref{lem.Propery_Projection}(ii) and \Cref{lem.Fejer_monotonicity_SiPM}, respectively. Then, the result follows then from \Cref{lem.ConvergenceRate_Fejermonoton}.
\end{proof}

We now proceed to prove the linear convergence of s-cCRM. First, we show  the linear rate for the most violated constraint control (Algorithm 2) related to \cref{eq.MostViolated1}, in view of \cref{eq.EB1}.

\begin{lemma}[Linear convergence of the distance for s-cCRM (Algorithm 2) under EB1]
    \label[lemma]{lem.Linear_convergence_cCRM_distanceSequecn_EB1}
    Let $C_1,C_2,\ldots,C_m$ be nonempty closed convex sets, and assume that $C\coloneqq \bigcap_{i=1}^m C_i\neq \emptyset$. Assume that \cref{eq.EB1} at $\Bar{z} \in C $.  Let $B$ be a ball centered at $\Bar{z}$ and contained in $V$. Let $z\in B$ and define $\beta \coloneqq \sqrt{1-\omega^2}$. Then,  
\begin{equation}
\label{eq.Linear_convergence_cCRM_distanceSequecn_EB1.1}
    \dist(T_{C_{r},C_{\ell}}(z),C) \leq  \beta^2 \dist(z,C),  
\end{equation}
where $\ell  \coloneqq \argmaxA\limits_{1\leq i \leq m} \{\dist(z, C_i)\} $ and  $ r \coloneqq \argmaxA\limits_{1\leq i \leq m} \{\dist(P_{C_{\ell}}(z),  {C_i}) \}$.

\end{lemma}

\begin{proof}
     By \cref{eq.LinearConvergence_SiPM_SePM_EB1.5}, we know that $Z_{C_{r},C_{\ell}}(z) \in B$. Using the definition of $\Bar{Z}_{C_{r},C_{\ell}}$, we have
    \begin{align}
        \dist(\Bar{Z}_{C_{r},C_{\ell}}(z),C) & = \dist\left(\frac{1}{2}\right(Z_{C_{r},C_{\ell}}(z)+P_{C_{\ell}} (Z_{C_{r},C_{\ell}}(z))\left),C\right) \\
        & \leq \frac{1}{2}\dist(Z_{C_{r},C_{\ell}}(z),C)   +\frac{1}{2}\dist(P_{C_{\ell}} (Z_{C_{r},C_{\ell}}(z)),C) \\
        & \leq \frac{1}{2}\beta^2 \dist(z,C)+ \frac{1}{2}\beta^2 \dist(z,C)\\
        & = \beta^2 \dist(z,C),\label{eq.Linear_convergence_cCRM_distanceSequecn_EB1.2}
    \end{align}
where the first inequality follows from the convexity of the distance function, and the second from \Cref{lem.Propery_Projection}(iv) and \cref{eq.LinearConvergence_SiPM_SePM_EB1.1}. 

By \Cref{lem.Fejermonotonicty_Circ}, we have 
\begin{equation}
\label{eq.Linear_convergence_cCRM_distanceSequecn_EB1.3}
    \| \pCRMOp_{C_{r},C_{\ell}}(z) -s\| \leq \|z-s\|,
\end{equation}
for any centralized point $z\in B$ and for any $s \in C$. By \Cref{lem.CentralizedProcedure} and \cref{eq.LinearConvergence_SiPM_SePM_EB1.5},  $\Bar{Z}_{C_{r},C_{\ell}}(z)$ is centralized with respect to $C_{\ell}$ and $C_{r}$ and belongs to $B$. So, taking $\Bar{Z}_{C_{r},C_{\ell}}(z)$, we get
\begin{equation}
    \label{Linear_convergence_cCRM_distanceSequecn_EB1.4}
    \| T_{C_{r},C_{\ell}}(z) -s\| = \norm{\pCRMOp_{C_{r},C_{\ell}}(\Bar{Z}_{C_{r},C_{\ell}}(z)) - s } \leq \|\Bar{Z}_{C_{r},C_{\ell}}(z)-s\|,
\end{equation}
where the inequality follows from \Cref{lem.FirmlyNonexpansiveness_Cirmcum}. Using the definition of distance between $T_{C_{r},C_{\ell}}(z)$ and $C$, we get
\begin{equation}
    \label{eq.Linear_convergence_cCRM_distanceSequecn_EB1.5}
    \dist(T_{C_{r},C_{\ell}}(z),C) \leq \|\Bar{Z}_{C_{r},C_{\ell}}(z)-s\|.
\end{equation}
Take $s=P_C(\Bar{Z}_{C_{r},C_{\ell}}(z))$, then
\begin{equation}
\label{eq.Linear_convergence_cCRM_distanceSequecn_EB1.6}
    \dist(T_{C_{r},C_{\ell}}(z),C) \leq \dist (\Bar{Z}_{C_{r},C_{\ell}}(z),C),
\end{equation}
for all $z\in B$. Combining \cref{eq.Linear_convergence_cCRM_distanceSequecn_EB1.2} and \cref{eq.Linear_convergence_cCRM_distanceSequecn_EB1.6}, we obtain
\begin{equation}
    \begin{split}
        \dist(T_{C_{r},C_{\ell}}(z),C) & \leq \beta^2 \dist(z,C),
    \end{split}
\end{equation}
which establishes the result.\end{proof}

Now, we are going to establish linear convergence of Algorithm 2, under \cref{eq.EB1}.
\begin{theorem}[Linear convergence of s-cCRM (Algorithm 2) under EB1]
\label{thm.LinearConvergence_cCRM_MostViolated1}
Let $C_1,\ldots,C_m\subset \re^n$ be  nonempty closed convex sets, and suppose that $C\coloneqq \bigcap_{i=1}^m C_i \neq \emptyset$. Assume that sequence $\{z^k\}$ is generated by s-cCRM with the most violated constraint control sequence (distance version) as in \cref{eq.MostViolated1}, starting from some $z^0\in \re^n$. Assume also that $\{z^k\}$ is an infinite sequence. If \cref{eq.EB1} holds at the limit $\Bar{z}$ of $\{z^k\}$, then $\{z^k\}$ converges to $\Bar{z}\in C$ R-linearly, with asymptotic constant bounded above $\beta^2$, where $\beta = \sqrt{1-\omega^2}$ and $\omega$ is the constant from \cref{def.EB1}.
\end{theorem}
\begin{proof}
    The convergence of $\{z^k\}$ to a point $\Bar{z} \in C$ follows from \Cref{thm.Convergence_MostViolated1}. Hence, for large enough $k$, $z^k$ belongs to the ball centered at $\Bar{z}$ and contained in $V$, whose existence is ensured in \cref{eq.EB1}.

We recall that the s-cCRM sequence is defined as $z^{k+1}=T_{C_{r(k)},C_{\ell(k)}}(z^k)$, so that it follows from \Cref{lem.Linear_convergence_cCRM_distanceSequecn_EB1} that
\begin{equation}
\label{eq.thm4}
    \frac{\dist(z^{k+1},C)}{\dist(z^{k},C)} \leq \beta^2.
\end{equation}
Since $\omega\in (0,1)$ implies that $\beta^2\in(0,1)$, it follows immediately from \cref{eq.thm4} that the scalar sequence $\{\dist(z^k,C)\}$ converges Q-linearly to $0$ with asymptotic constant bounded above by $\beta^2$.

Finally, recall that sequence $\{z^k\}$ is Fej\'er monotone with respect to $C$, due to \Cref{cor.Fejermonotonicity_cCRM_m-sets_CFP.1}. The R-linear convergence of $\{z^k\}$ to some point in $C$ and the value of the upper bound of the asymptotic follow from \Cref{lem.ConvergenceRate_Fejermonoton}.\end{proof}
%------------------------------------------------

\subsection{Linear convergence of Algorithm 3}

Before proving the linear convergence of s-cCRM (Algorithm 3) under \cref{eq.EB2}, we need some lemmas about the relationship between \cref{eq.EB1} and \cref{eq.EB2}. First we provide a bound for the norm of subdifferentials under \cref{eq.EB2}.

\begin{lemma}[Bound of subdifferentials under EB2]
\label[lemma]{lem.Subgradient_is_LocallyBounded}
Let $C_1,\ldots,C_m\subset \re^n$ be  nonempty closed convex sets, and suppose that $C\coloneqq \bigcap_{i=1}^m C_i \neq \emptyset$.
    Assume that  \cref{eq.EB2} holds at a point $\Bar{z}\in C$. Then, there exists a ball $B$ centered at $\Bar{z}$ and a constant $\epsilon >0$ such that
    \begin{equation}       \label{eq.Subgradient_is_LocallyBounded}
        \|v_i(u)\| \leq \epsilon,
    \end{equation}
    for all $u\in B$, all $v_i(u)\in \partial f_i(u)$, and all $i\in\{1,\ldots,m\}$. 
\end{lemma}
\begin{proof}
    Take an arbitrary point $u\in \re^n$, and let $v_i(u)\in \re^n$ be any subgradient of $f_i$ at $u$. 
    Now, recall that the convex function $f_i$ is locally Lipschitz continuous in $\re^n$ for all $i \in \{1,\ldots,m\}$. Using the fact that the subdifferentials of $f_i$'s are locally bounded in $\re^n$, we establish the result.    
\end{proof}

\begin{lemma}[Relation betweem EB1 and EB2]
\label[lemma]{lema.CorollaryOfEB2}
Let $C_1,\ldots,C_m\subset \re^n$ be  nonempty closed convex sets, and suppose that $C\coloneqq \bigcap_{i=1}^m C_i \neq \emptyset$. Assume that \cref{eq.EB2} holds at point $\Bar{z}\in C$. Then, there exists $\epsilon >0$ such
\begin{equation}
\label{eq.CorollaryOfEB2.1}
    f_i(z) \leq \epsilon \dist(z,C_i),
\end{equation}
for all $i\in\{1,\ldots,m\}$, and all $z$ in the neighborhood $V$ of $\Bar{z}$ as defined in \cref{def.EB2}. Moreover,
\begin{equation}
    \label{eq.CorollaryOfEB2.2}
    \dist(z,C)\leq \frac{\epsilon}{\omega} \dist(z,C_{\ell}),
\end{equation}
where $\ell  \coloneqq  \argmaxA\limits_{1\leq i \leq m} \{f_i(z)\}$.
\end{lemma}
\begin{proof}
    By \Cref{lem.Subgradient_is_LocallyBounded}, there exists a ball $B$ contained in $V$, centered at $\Bar{z}$, and a constant $\epsilon$ such that
\begin{equation}
    \label{eq.CorollaryOfEB2.3}
    \|v_i(u)\| \leq \epsilon,
\end{equation}
for all $v_i(u) \in \partial f_i(u)$, and for all $u\in B$, and for each index $i\in\{1,\ldots,m\}$. Take any $z\in B$ and let $z_i \coloneqq  P_{C_i}(z)$. Using \Cref{lem.MVT_ConvexFunction}, we get 
\begin{equation}
\label{eq.CorollaryOfEB2.4}
    f_i(z) = f_i(z_i) + \scal{v_i(u_i)}{z-z_i},
\end{equation}
for some $u_i$ in the line segment between $z$ and $z_i$ and, $v_i(u_i) \in \partial f(u_i)$. Since $z_i \in C_i$, we have $f_i(z_i)=0$, and it follows from \cref{eq.CorollaryOfEB2.4} that
\begin{equation}
\label{eq.CorollaryOfEB2.5}
    f_i(z) \leq \|v_i(u_i)\|\|z-z_i\|=\|v_i(u_i)\| \dist(z,C_i).
\end{equation}
By the nonexpansiveness of the orthogonal projection, we have
\begin{equation}
    \label{eq.CorollaryOfEB2.6}
    \|P_{C_i}(z) - P_{C_i}(\Bar{z})\| \leq \|z-\Bar{z}\|.
\end{equation}
The definition of $z_i$ and $\Bar{z} \in C$, yields
\begin{equation}
    \label{eq.CorollaryOfEB2.7}
    \|z_i-\Bar{z}\| \leq \|z-\Bar{z}\|,
\end{equation}
which shows that $z_i \in B$. Hence, $u_i\in B$, by the convexity of the ball. In view of \cref{eq.CorollaryOfEB2.3} and \cref{eq.CorollaryOfEB2.5}, we have
\begin{equation}
\label{eq.CorollaryOfEB2.8}
    f_i(z) \leq \epsilon \dist(z,C_i),
\end{equation}
for each $i=1,\ldots,m$ and each $z\in B$. Therefore,
\begin{equation}
    \label{eq.CorollaryOfEB2.9}
    f_{\ell}(z) \leq \epsilon \dist(z,C_{\ell}),
\end{equation}
so that, in view of \cref{eq.EB2}, it holds that 
% \begin{equation}
%      \dist(z,C) \leq \frac{1}{\omega} \max_{1\leq i\leq m}\{f_i(z)\} \leq \frac{1}{\omega} \max_{1\leq i\leq m}\{Q_i \dist(z,C_i)\}
% \end{equation}
\begin{equation}
    \label{eq.CorollaryOfEB2.10}
    \dist(z,C) \leq \frac{\epsilon}{\omega} f_{\ell}(z) \leq \frac{\epsilon}{\omega} \dist(z,C_{\ell}),
\end{equation}
which establishes the result.\end{proof}

\begin{remark}
    With arguments very similar to those used in the previous lemma, together with the nonexpansiveness of projection operator, we can easily get
\begin{equation}
\label{eq.remark_EB2}
    \dist(P_{C_{\ell}}(z),C) \leq \frac{\epsilon}{\omega} \dist(P_{C_{\ell}}(z),C_{r}),
\end{equation}
where $r \coloneqq \argmaxA\limits_{1\leq i \leq m} \{f_i(P_{C_{\ell}}(z)) \}$.
\end{remark} 
% \textbf{Remark:} By the last lemma we know for each $f_i$ there exist a ball $\Tilde{B}_i$ and a constant $Q_i$ such that (4.22) holds. For the convenience of proof, we define $\Tilde{B}=\bigcap_{i=1}^m \Tilde{B}_i$ and $Q = \max_{1\leq i\leq m}\{Q_i\}$.
\begin{lemma}[Linear convergence of the distance for SePM and SiPM under EB2]
\label[lemma]{lem.LinearConvergence_SePM_SiPM_EB2_distanceSequence}
Let $C_1,\ldots,C_m\subset \re^n$ be  nonempty closed convex sets, and assume $C\coloneqq \bigcap_{i=1}^m C_i \neq \emptyset$
Assume that \cref{eq.EB2} holds at $\Bar{z} \in C $, and that $B$ is a ball centered at $\Bar{z}$ and contained in $V$. Define $\beta\coloneqq  \sqrt{1-(\frac{\omega}{\epsilon})^2}$, with $\omega$ being the constant in \cref{def.EB2}, and $\epsilon$ being the constant in \Cref{lema.CorollaryOfEB2}. Then,
\begin{equation}
\label{eq.LinearConvergence_SePM_SiPM_EB2_distanceSequence.1}
    \dist(Z_{C_{r},C_{\ell}}(z),C) \leq \beta^2 \dist(z,C),
\end{equation}
for all $z\in B$, and
\begin{equation}
\label{eq.LinearConvergence_SePM_SiPM_EB2_distanceSequence.2}
    \dist(\Tilde{Z}_{C_{r},C_{\ell}}(z),C) \leq \frac{1+\beta}{2} \dist(z,C),
\end{equation}
for all $z\in B$, where $\ell  \coloneqq \argmaxA\limits_{1\leq i \leq m} \{f_i(z)\}$ and  $r \coloneqq \argmaxA\limits_{1\leq i \leq m} \{f_i(P_{C_{\ell}}(z)) \}$. 
\end{lemma}
\begin{proof}
We start with proving \cref{eq.LinearConvergence_SePM_SiPM_EB2_distanceSequence.1}.
Take $z\in B$, and note that
\begin{align}
        {\dist}^2(z,C) & = \| z-P_C(z)\|^2 \\  & \geq \|P_{C_{\ell}}(z) - z \|^2 + \|P_{C_{\ell}}(z) - P_C(z)\|^2 \\
        & \geq  {\dist}^2(z,C_{\ell}) +{\dist}^2 (P_{C_{\ell}}(z),C) \\
        & \geq {\dist}^2 (P_{C_{\ell}}(z),C) + \frac{\omega^2}{\epsilon^2} {\dist}^2(z,C), 
        % & \geq {\dist}^2 (P_{C_{\ell}}(z),C) + \frac{\omega}{Q} {\dist}^2(z,C)
        \label{eq.LinearConvergence_SePM_SiPM_EB2_distanceSequence.3}
\end{align}
using \Cref{lem.Propery_Projection}(iii) in the first inequality, the definition of orthogonal projection in the second one, and \cref{eq.CorollaryOfEB2.2} in the third one. Hence, 
\begin{equation}
\label{eq.LinearConvergence_SePM_SiPM_EB2_distanceSequence.4}
    \dist (P_{C_{\ell}}(z),C) \leq \sqrt{1-\frac{\omega^2}{\epsilon^2}}\dist(z,C) = \beta \dist(z,C).
\end{equation}
By \cref{eq.LinearConvergence_SiPM_SePM_EB1.5}, we get $P_{C_{r}}(P_{C_{\ell}}(z)) \in B$. Therefore, 
\begin{equation}
\label{eq.LinearConvergence_SePM_SiPM_EB2_distanceSequence.5}
    \begin{split}
        {\dist}^2(P_{C_{\ell}}(z),C) & = \|P_{C_{\ell}}(z) - P_C(P_{C_{\ell}}(z))\|^2 \\ 
        & \geq \|P_{C_{r}}(P_{C_{\ell}}(z)) - P_C(P_{C_{\ell}}(z))\|^2  \\
        & + \|P_{C_{r}}(P_{C_{\ell}}(z)) - P_{C_{\ell}}(z)\|^2 \\
        & \geq {\dist}^2(P_{C_{r}}(P_{C_{\ell}}(z)),C)+ {\dist}^2(P_{C_{\ell}}(z),C_{r})   \\
        % & \geq {\dist}^2(P_{C_{\ell}}(P_{C_{\ell}}(z)),C)  + \frac{\omega}{Q_{r}} \dist(P_{C_{\ell}}(z),C) \\
        & \geq {\dist}^2(P_{C_{r}}(P_{C_{\ell}}(z)),C)  + \frac{\omega^2}{\epsilon^2} {\dist}^2(P_{C_{\ell}}(z),C),
    \end{split}
\end{equation}
using \Cref{lem.Propery_Projection}(iii) in the first inequality, the definition of orthogonal projection in the second one, and \cref{eq.remark_EB2} in the third one. Thus, we obtain
\begin{equation}
\label{eq.LinearConvergence_SePM_SiPM_EB2_distanceSequence.6}
    \dist(P_{C_{r}}(P_{C_{\ell}}(z)),C) \leq \beta \dist(P_{C_{\ell}}(z)),C).
\end{equation}
Together with \cref{eq.LinearConvergence_SePM_SiPM_EB2_distanceSequence.4}, we have 
\begin{equation}
\label{eq.LinearConvergence_SePM_SiPM_EB2_distanceSequence.7}
        \dist(Z_{C_{r},C_{\ell}}(z),C) = \dist(P_{C_{r}}(P_{C_{\ell}}(z)),C) \leq \beta^2 \dist (z,C).
\end{equation}
Next, we prove \cref{eq.LinearConvergence_SePM_SiPM_EB2_distanceSequence.2}. By the nonexpansiveness of $P_{C_{r}}$, we have
\begin{equation}
\label{eq.LinearConvergence_SePM_SiPM_EB2_distanceSequence.9}
    \dist(P_{C_{r}}(z),C) \leq \dist(z,C).
\end{equation}
Note that
\begin{align}
        \dist(\Tilde{Z}_{C_{r},C_{\ell}}(z),C) & = \dist(\frac{1}{2}[P_{C_{\ell}}(z) + P_{C_{r}}(z)],C) \\
        & \leq \frac{1}{2} [\dist(P_{C_{\ell}}(z),C) + \dist(P_{C_{r}}(z),C)] \\
        & \leq \left(\frac{1+\beta}{2}\right) \dist(z,C).\label{eq.LinearConvergence_SePM_SiPM_EB2_distanceSequence.8}
\end{align}
The first inequality holds by the convexity of the distance function, and the second one follows by \cref{eq.LinearConvergence_SePM_SiPM_EB2_distanceSequence.9} and \cref{eq.LinearConvergence_SePM_SiPM_EB2_distanceSequence.4}.\end{proof}

\begin{corollary}[Linear convergence of SePM and SiPM under {EB2}]
\label[corollary]{corollary_SePM_SiPM_EB2.2}
Let $C_1,\ldots,C_m\subset \re^n$ be  nonempty closed convex sets, and assume that $C\coloneqq \bigcap_{i=1}^m C_i\neq \emptyset$. Let $\{s^k\}$ and $\{y^k\}$ be sequences generated by SePM and SiPM, starting from some $s^0 \in \re^n$ and $y^0\in \re^n$, respectively. Assume also $\{s^k\}$ and $\{y^k\}$ are both infinite sequences. If \cref{eq.EB2} holds at the limit point $\Bar{s}$ of $\{s^k\}$, $\Bar{y}$ of $\{y^k\}$, then the sequences $\{s^k\}$, $\{y^k\}$ converge R-linearly, with asymptotic constants bounded above by $\beta^2$, $\frac{1+\beta}{2}$ respectively, where $\beta= \sqrt{1-(\frac{\omega}{\epsilon})^2}$, $\omega$ is the constant in \cref{def.EB2}, and $\epsilon$ is the constant in \Cref{lema.CorollaryOfEB2}.
\end{corollary}
\begin{proof}
    We invoke again  \cite[Corollary 3.3(i)]{Bauschke:1996} and  \cite[Theorem 3]{DePierro:1985}, to get the convergence of $\{s^k\}$ and $\{y^k\}$ to points $\Bar{s} \in C$ and $\Bar{y} \in C$, respectively. Hence,   $s^k$ belongs to a ball centered at $\Bar{s}$ contained in $V$, and $y^k$ belongs to a ball centered at $\Bar{y}$ contained in $V$ and, for large enough $k$.

    In view of the definitions of the SePM and SiPM sequences, from \Cref{lem.LinearConvergence_SePM_SiPM_EB2_distanceSequence} we derive that
    \begin{equation}
        \frac{\dist(s^{k+1},C)}{\dist(s^k,C)} \leq \beta^2,\quad \text{ and } \quad \frac{\dist(y^{k+1},C)}{\dist(y^k,C)} \leq \frac{1+\beta}{2}.
    \end{equation}
The limits of the above inequalities (with $k\to \infty$), together with \Cref{def.LinearConvergence}, yield that the sequences $\{\dist(s^k, C)\}$ and $\{\dist(y^k,C)\}$ converge {Q-linearly} to $0$, with asymptotic constants given by $\beta^2$ and $\frac{1+\beta}{2}$, respectively, where $\beta= \sqrt{1-(\frac{\omega}{\epsilon})^2}$, $\omega$ is the constant in \cref{eq.EB2}, and $\epsilon$ is the constant in \Cref{lema.CorollaryOfEB2}.

    Remind that $\{s^k\}$ and $\{y^k\}$ are Fej\'er monotone with respect to $C$ (by \Cref{lem.Propery_Projection}(ii) and \Cref{lem.Fejer_monotonicity_SiPM}, respectively). Hence, the result follows from \Cref{lem.ConvergenceRate_Fejermonoton}.\end{proof}

    Now, we present a lemma that allows us to prove the linear convergence of Algorithm 3, under \cref{eq.EB2}. 
\begin{lemma}[Linear convergence of distance for s-cCRM (Algorithm 3) under EB2]
    \label[lemma]{lem.x}
    Suppose $C_1,C_2,\ldots,C_m\subset \re^n $ nonempty closed convex sets and assume $C\coloneqq \bigcap_{i=1}^m C_i \neq \emptyset$
    Assume that \cref{eq.EB2} at $\Bar{z} \in C $. 
    Let $B$ be a ball centered at $\Bar{z}$ and contained in $V$. Let $z\in B$ and define $\beta\coloneqq  \sqrt{1-(\frac{\omega}{\epsilon})^2}$, with $\omega$ being the constant in \cref{def.EB2}, and $\epsilon$ being the constant in \Cref{lema.CorollaryOfEB2}. Then,
\begin{equation}
\label{eq.x.1}
    \dist(T_{C_{r},C_{\ell}}(z),C) \leq \beta^2 \dist(z,C),  
\end{equation}
for all $z\in B$. 
\end{lemma}
\begin{proof}
    By \cref{eq.LinearConvergence_SiPM_SePM_EB1.5}, we know that $Z_{C_{r},C_{\ell}}(z) \in B$. Using the definition of $\Bar{Z}_{C_{r},C_{\ell}}$, we have
\begin{align}
        \dist(\Bar{Z}_{C_{r},C_{\ell}}(z),C) & = \dist\left(\frac{1}{2}[Z_{C_{r},C_{\ell}}(z)+P_{C_{\ell}} (Z_{C_{r},C_{\ell}}(z))],C\right) \\
        &  \leq \frac{1}{2}\dist(Z_{C_{r},C_{\ell}}(z),C) +\frac{1}{2}\dist(P_{C_{\ell}} (Z_{C_{r},C_{\ell}}(z)),C)  \\
        & \leq \frac{1}{2}\beta^2 \dist(z,C)+ \frac{1}{2}\beta^2 \dist(z,C)\\
        & = \beta^2 \dist(z,C),\label{eq.x.2}
\end{align}
where the first inequality follows from the convexity of the distance function, and the second from \Cref{lem.Propery_Projection}(iv) and \cref{eq.LinearConvergence_SePM_SiPM_EB2_distanceSequence.1}. Combining \cref{eq.x.2} and \cref{eq.Linear_convergence_cCRM_distanceSequecn_EB1.6}, we establish the result.\end{proof}

We finalize this section stating and proving the linear convergence of s-cCRM with the most violated constraint control sequence (function value version) as in \cref{eq.MostViolated2}, under EB2. 

\begin{theorem}[Linear convergence of s-cCRM (Algorithm 3) under EB2]
\label{thm.LinearConvergence_s-cCRM_MostViolated2}
Let $C_1,\ldots,C_m\subset \re^n$ be  nonempty closed convex sets, and suppose that $C\coloneqq \bigcap_{i=1}^m C_i \neq \emptyset$. Let the sequence $\{z^k\}$ be generated by s-cCRM with the most violated constraint control sequence (function value version) as in \cref{eq.MostViolated2}, starting from some $z^0\in \re^n$. Assume also that $\{z^k\}$ is an infinite sequence. If \cref{eq.EB2} holds at the limit $\Bar{z}$ of $\{z^k\}$, then $\{z^k\}$ converges to $\Bar{z}\in C$ R-linearly, with asymptotic constant bounded above $\beta^2$, where $\beta= \sqrt{1-(\frac{\omega}{\epsilon})^2}$, $\omega$ is the constant in \cref{def.EB2}, and $\epsilon$ is the constant in \Cref{lema.CorollaryOfEB2}.

\end{theorem}
\begin{proof}
    Convergence of $\{z^k\}$ to a point $\Bar{z} \in C$ follows from \Cref{thm.Convergence_MostViolated2}. Hence, for large enough $k$, $z^k$ belongs to the ball centered at $\Bar{z}$ and contained in $V$, whose existence is ensured in \cref{eq.EB2}.

We recall that the s-cCRM sequence is defined as $z^{k+1}=T_{C_{r(k)},C_{\ell(k)}}(z^k)$, so that it follows from \Cref{lem.x} that
\begin{equation}
\label{eq.LinearConvergence_cCRM_MostViolated2.1}
    \frac{\dist(z^{k+1},C)}{\dist(z^{k},C)} \leq \beta^2.
\end{equation}
Since $\beta^2\in(0,1)$, it follows immediately from \cref{eq.LinearConvergence_cCRM_MostViolated2.1} that the scalar sequence $\{\dist(z^k,C)\}$ converges Q-linearly to zero with asymptotic constant bounded above by $\beta^2$.

Finally, recall that the sequence $\{z^k\}$ is Fej\'er monotone with respect to $C$, due to \Cref{cor.Fejermonotonicity_cCRM_m-sets_CFP.1}. The R-linear convergence of $\{z^k\}$ to some point in $C$ and the value of the upper bound of the asymptotic follow from \Cref{lem.ConvergenceRate_Fejermonoton}.\end{proof}

%--------------Superlinear Convergence----------------
\subsection{Superlinear convergence of s-cCRM}

In this subsection we prove superlinear convergence of Algorithms 2 and 3 versions of s-cCRM, assuming a Slater condition and a smoothness assumption: the boundaries of the sets $C_i$ are
differentiable manifolds (of codimension $1$, due to the Slater condition) near the limit of the sequence.
First, we need a lemma about differentiable manifolds.
\begin{lemma}[Dimension of differentiable manifolds boundaries {\cite[Thm.~24.3]{Munkres:1997}}]
    \label[lemma]{lem.Dimension_Manifold}
    Let $M$ be a $k$-dimensional manifold in $\re^n$, of class $\mathcal{C}^p$. If the boundary of $M$, $\bound(M)$, is nonempty, then $\bound(M)$ is a $(k-1)$-dimensional manifold without boundary in $\re^n$, of class $\mathcal{C}^p$.
\end{lemma}

Now we can demonstrate the superlinear convergence of s-cCRM. We begin by the most violated constraint control sequence (distance version) as in \cref{eq.MostViolated1} (Algorithm 2).

\begin{lemma}[Superlinear convergence of the distance for s-cCRM (Algorithm 2)]
    \label{lem.Superlinear1}
    Let $C_1,\ldots,C_m\subset \re^n$ be  nonempty closed convex sets, and suppose that $C\coloneqq \bigcap_{i=1}^m C_i \neq \emptyset$. Let sequence the $\{z^k\}$ be generated by s-cCRM with the most violated constraint control sequence (distance version) as in \cref{eq.MostViolated1}, starting from some $z^0\in \re^n$, and converging to a point $\Bar{z}\in C$. Assume that the {interior} of $C$ is nonempty and that the boundaries of $C_i$ are differentiable manifolds in a neighborhood of $\Bar{z}$ for each $i=1,\ldots,m$. Then, the scalar sequence $\{\dist(z^k,C)\}$ converges to zero superlinearly.
\end{lemma}

\begin{proof}
    It is trivial if sequence $\{z^k\}$ is finite. Hence, let's assume that it is infinite.

    In order to prove the superlinear convergence rate, \emph{i.e.},
    \begin{equation}
        \label{eq.Superlear_Most1control1}
        \lim_{k \rightarrow \infty} \frac{\dist(z^{k+1},C)}{\dist(z^{k},C)}=\lim_{k \rightarrow \infty} \frac{\dist(T_{C_{r(k)},C_{\ell(k)}}(z^k),C)}{\dist(z^{k},C)}=0
    \end{equation}
    it suffices to show that
    \begin{equation}
        \label{eq.Superlear_Most1control2}
        \lim_{k \rightarrow \infty} \frac{\dist(T_{C_{r(k)},C_{\ell(k)}}(z^k),C)}{\dist(\Bar{Z}_{C_{r(k)},C_{\ell(k)}}(z^k),C)} = 0,
    \end{equation}
    because, by \cref{eq.SupplyToConvergenceAnalysis.2} and the nonexpansiveness of orthogonal projection, we know that $\dist(\Bar{Z}_{C_{r(k)},C_{\ell(k)}}(z^k),C) \leq \dist(z^k,C)$, so that \cref{eq.Superlear_Most1control1} follows from  \cref{eq.Superlear_Most1control2} immediately. 

    We claim that the assumption $\inte(C) \neq \emptyset$ implies \cref{eq.EB1}. Indeed, by Corollary 5.14 in~\cite{Bauschke:1996}, there exists $\hat{k}\in N$ such that
    \begin{equation}
    \label{eq.Superlear_Most1control3}
        \omega \dist(z^k,C) \leq \max_{1\leq i\leq m} \{\dist(z^k,C_i)\} = \dist(z^k,C_{\ell(k)}),
    \end{equation}
    for all $\hat{k} \geq k$.
    In addition, the nonemptyness of the interior of $C$, together with the hypothesis that the boundaries of $C_i$ are locally differentiable manifolds, proves that these manifolds have dimension $n-1$, in view of \Cref{lem.Dimension_Manifold}. 
    
 We remark that, when $M\subset \mathds{R}^n$ is a differentiable manifold of dimension  $n-1$, we have that if  $\bar z$ belongs to $M$ and $z\in \re^n$ lies on the tangent hyperplane to $M$ at $\bar z$, denoted here by $T_M(\bar z)$, then
\begin{equation}\label{v1}
\lim_{\substack{z\to\bar z \\ z\in T_M(\bar z) }}\frac{\dist(z,M)}{\lV z-\bar z\rV}=0.
\end{equation}
  This limit is discussed in more detail in \cite[Eq.~3.32]{Behling:2021}.
    
    Consider the $(n-1)$ dimensional hyperplanes $H_{C_{\ell(k)}}^k$ and $H_{C_{r(k)}}^k$, which are tangent to the manifolds, respectively at \[P_{C_{\ell(k)}}(\Bar{Z}_{C_{r(k)},C_{\ell(k)}}(z^k)) \text{ and } P_{C_{r(k)}}(\Bar{Z}_{C_{r(k)},C_{\ell(k)}}(z^k)).\] 
    Now, note that   $z^{k+1}$ is the circumcenter of $\{z^k, R_{C_{\ell(k)}}(z^k), R_{C_{r(k)}}(z^k)\}$. So, $z^{k+1}$ lies in the intersection of the bisectors passing by $P_{C_{\ell(k)}}(z^k) \in H_{C_{\ell(k)}}^k$, and $P_{C_{r(k)}}(z^k)\in H_{C_{r(k)}}^k$, respectively. Each bisector is contained in the hyperplane $H_{C_{\ell(k)}}^k$ or $H_{C_{r(k)}}^k$, and hence $z^{k+1} \in  H_{C_{\ell(k)}}^k \cap H_{C_{r(k)}}^k$. 
    Therefore,  we have, in view of \eqref{v1}, 
    \begin{equation}
        \label{eq.Superlear_Most1control4}
        \lim_{k \rightarrow \infty} \frac{\dist(z^{k+1},C_{\ell(k)})}{\|z^{k+1}-P_{C_{\ell(k)}}(\Bar{Z}_{\ell(k),r(k)}(z^k))\|}=0.
    \end{equation}

    Now, using the nonexpansiveness of projections onto $H_{C_{\ell(k)}}$ and $H_{C_{r(k)}}$, we get
    \begin{equation}
        \label{eq.Superlear_Most1control6}
        \|P_{H_{C_{\ell(k)}}}(z^{k+1}) - P_{H_{C_{\ell(k)}}}(\Bar{Z}_{C_{r(k)},C_{\ell(k)}}(z^k))\| \leq \| z^{k+1} -\Bar{Z}_{C_{r(k)},C_{\ell(k)}}(z^k)\|.  
    \end{equation}
    Since $z^{k+1}\in H_{C_{\ell(k)}}$ and $P_{H_{C_{\ell(k)}}}(\Bar{Z}_{C_{r(k)},C_{\ell(k)}}(z^k)) = P_{C_{\ell(k)}}(\Bar{Z}_{C_{r(k)},C_{\ell(k)}}(z^k))$, we get
    \begin{equation}
        \label{eq.Superlear_Most1control7}
        \|z^{k+1}-P_{C_{\ell(k)}}(\Bar{Z}_{C_{r(k)},C_{\ell(k)}}(z^k)) \|\leq \| z^{k+1} - \Bar{Z}_{C_{r(k)},C_{\ell(k)}}(z^k)\|.
    \end{equation}    
    By \Cref{lem.FirmlyNonexpansiveness_Cirmcum}, it holds that $\|z^{k+1} - \Bar{Z}_{C_{r(k)},C_{\ell(k)}}(z^k)\| \leq \dist(\Bar{Z}_{C_{r(k)},C_{\ell(k)}}(z^k),C)$, which combined with  \cref{eq.Superlear_Most1control7}, implies that
    \begin{equation}
    \label{eq.Superlear_Most1control9}
        \|z^{k+1}-P_{C_{\ell(k)}}(\Bar{Z}_{C_{r(k)},C_{\ell(k)}}(z^k)) \| \leq \dist(\Bar{Z}_{C_{r(k)},C_{\ell(k)}}(z^k),C).
    \end{equation}
    % and
    % \begin{equation}
    % \label{eq.Superlear_Most1control10}
    %     \|z^{k+1}-P_{C_{r(k)}}(\Bar{Z}_{C_{r(k)},C_{\ell(k)}}(z^k)) \| \leq \dist(\Bar{Z}_{C_{r(k)},C_{\ell(k)}}(z^k),C)
    % \end{equation}
    Thus, from \cref{eq.Superlear_Most1control4}, it follows that
    \begin{equation}
        \label{eq.Superlear_Most1control11}
        \lim_{k \rightarrow \infty} \frac{\dist(z^{k+1},C_{\ell(k)})}{\dist(\Bar{Z}_{C_{r(k)},C_{\ell(k)}}(z^k),C)}=0. 
    \end{equation}

    Moreover, \cref{eq.Superlear_Most1control3} yields 
    \begin{equation}
        \label{eq.Superlear_Most1control13}
        \omega \frac{\dist(z^{k+1},C)}{\dist(\Bar{Z}_{C_{r(k)},C_{\ell(k)}}(z^k),C)} \leq \frac{\dist(z^{k+1},C_{\ell(k)})}{\dist(\Bar{Z}_{C_{r(k)},C_{\ell(k)}}(z^k),C)}.
    \end{equation}
    Taking limits now as $k\rightarrow \infty$, we get \cref{eq.Superlear_Most1control2} and the proof is completed.\end{proof}

    The next key result allows us to use the superlinear rate of a scalar distance sequence to prove the superlinear rate of the underlying sequence.

    \begin{lemma}[Fejér monotonicity and superlinear convergence {\cite[Prop.~3.12]{Behling:2021}}]
           \label[lemma]{lem.Superlinear_convergence_FejerMonoton}
    Take a sequence $\{z^k\} \subset \re^n$ which is Fej\'er monotone with respect to the closed convex set $M\subset \re^n$. If the scalar sequence $\{\dist(z^k,M)\}$ converges superlinearly to $0$, then $\{z^k\}$ converges superlinearly to a point $\Bar{z}\in M$. 
    \end{lemma}

Next, we present the superlinear convergence result for Algorithm 2. 

\begin{theorem}[Superlinear convergence of s-cCRM (Algorithm 2)]
\label{thm.Superlinear1}
Let $C_1,\ldots,C_m\subset \re^n$ be  nonempty closed convex sets, and suppose that $C\coloneqq\bigcap_{i=1}^m C_i$ is nonempty. Let $\{z^k\}$ be generated by s-cCRM with the most violated control sequence (distance version) as in  \cref{eq.MostViolated1}, starting from some $z^0\in \re^n$, and converging to a point $\Bar{z}\in C$. Assume that the interior of $C$ is nonempty, and that the boundary of $C_i$ is a differentiable manifold in a neighborhood of $\Bar{z}$ for each $i=1,\ldots, m$. Then, $\{z^k\}$ converges to $\Bar{z}$ superlinearly. 
\end{theorem}
\begin{proof}
    The result is a direct consequence of the Fej\'er monotonicity of $\{z^k\}$ with respect to $C$ given in \Cref{cor.Fejermonotonicity_cCRM_m-sets_CFP.1}, together with \Cref{lem.Superlinear1} and \Cref{lem.Superlinear_convergence_FejerMonoton}.
\end{proof}

In the next two results we prove the superlinear convergence of Algorithm 3 under the same assumptions.

\begin{lemma}[Superlinear convergence of the distance for s-cCRM (Algorithm 3)]
    \label[lemma]{lem.Superlinear2}
    Let $C_1,\ldots,C_m\subset \re^n$ be  nonempty closed convex sets, and suppose that $C\coloneqq \bigcap_{i=1}^m C_i \neq \emptyset$. Let sequence $\{z^k\}$ be generated by s-cCRM with the most violated constraint control sequence (function value version) as in \cref{eq.MostViolated2}, starting from some $z^0\in \re^n$, and converging to a point $\Bar{z}\in C$. Assume that the {interior} of $C$ is nonempty and that the boundaries of $C_i$ are differentiable manifolds in a neighborhood of $\Bar{z}$ for each $i=1,\ldots,m$. Also, assume that \cref{eq.EB2} holds at $\Bar{z}$. Then, the scalar sequence $\{\dist(z^k,C)\}$ converges to zero superlinearly.
\end{lemma}
\begin{proof}
    Assume that $\{z^k\}$ is infinite. Using the definition of \cref{eq.EB1} and the fact that $z^k\rightarrow \Bar{z}$, we conclude that there exists a positive integer $\hat{k}$ such that 
    \begin{equation}
     \label{eq.superlinear_theorem2.1}
        \dist (z^k ,C) \leq \frac{\omega}{\epsilon} \dist(z^k,C_{\ell(k)}),
    \end{equation}
    for all $k\geq \hat{k}$, where $\omega$ and $\epsilon$ are the constants in EB2. By \cref{eq.superlinear_theorem2.1}, we get 
    \begin{equation}
        \label{eq.superlinear_theorem2.2}
         \frac{\dist(z^{k+1},C)}{\dist(\Bar{Z}_{C_{r(k)},C_{\ell(k)}}(z^k),C)} \leq \frac{\omega \dist(z^{k+1},C_{\ell(k)})}{\epsilon \dist(\Bar{Z}_{C_{r(k)},C_{\ell(k)}}(z^k),C)}.
    \end{equation}
    Combining \cref{eq.Superlear_Most1control11} and \cref{eq.superlinear_theorem2.2}, we get \cref{eq.Superlear_Most1control2} which is sufficient to guarantee \cref{eq.Superlear_Most1control1}. Then the result holds.
\end{proof}

\begin{theorem}[Superlinear convergence of s-cCRM (Algorithm 3)]
\label{thm.Superlinear2}
Let $C_1,\ldots,C_m\subset \re^n$ be  nonempty closed convex sets, and suppose that $C\coloneqq\bigcap_{i=1}^m C_i$ is nonempty. Let $\{z^k\}$ be generated by s-cCRM with the most violated control sequence \cref{eq.MostViolated2} starting from some $z^0\in \re^n$, and converging to a point $\Bar{z}\in C$. Assume that the interior of $C$ is nonempty, and that the boundary of $C_i$ is a differentiable manifold in a neighborhood of $\Bar{z}$ for each $i=1,\ldots, m$. Also, assume that \cref{eq.EB1} holds at $\Bar{z}$. Then, $\{z^k\}$ converges to $\Bar{z}$ superlinearly. 
\end{theorem}
\begin{proof}
    The result is a direct consequence of the Fej\'er monotonicity of $\{z^k\}$ with respect to $C$ given in \Cref{cor.Fejermonotonicity_cCRM_m-sets_CFP.1}, together with \Cref{lem.Superlinear_convergence_FejerMonoton,lem.Superlinear2}.
\end{proof}

\section{Numerical Experiments}\label{sec:NumericalExperiments}
In this section, we present the results of the computational experiments  comparing s-cCRM with SePM (referred as \texttt{SePM})  and  CRM-Prod (presented in \cref{eq:CRMProd} and denoted by \texttt{CRMprod}). For s-cCRM we consider Algorithm 1 (denoted by \texttt{Alg1}) and Algorithm 3 (designated as \texttt{Alg3}). For \texttt{Alg1} we use the cyclic control sequence given in \cref{eq:cyclic_control}. In view of \Cref{lema.CorollaryOfEB2}, we do not present results concerning Algorithm 2, since Algorithm 2 is somehow equivalent to Algorithm 3 in the presence of error bound; see \cref{eq.CorollaryOfEB2.1}.

We apply the aforementioned four methods to the problem of finding a point in the intersection of $m$ ellipsoids, \emph{i.e.}, 
\begin{equation}
\label{eq.Ellipsoid1}
   \text{ find } x^* \in \bigcap_{i=1}^m \xi_i.
\end{equation}
Here, each ellipsoid $\xi_i$ is a set given by
\begin{equation}
\label{eq.Ellipsoid2}
    \xi_i \coloneqq  \{x\in \re^n\mid f_i(x) \leq 0\},\text{ for }i=1,2,\ldots,m
\end{equation}
with $f_i:\re^n\rightarrow \re$ defined as
\begin{equation}
    f_i(x) = \scal{x}{A_ix}+2\scal{x}{b^i} - c_i.
\end{equation}
We consider $A_i$ a symmetric positive definite matrix, $b^i$  a vector, and $c_i$  a positive scalar, for each $i=1,\ldots,m$.

To construct the ellipsoids we follow the steps of \cite{Behling:2021b}. First, we form the ellipsoid $\xi_1$ by generating a matrix $A_1$ of in the form of $A_1 = \gamma \Id+B_1^\top B_1$ with $B_1\in \re^{n\times n}$, $\gamma\in \re_{++}$. Matrix $B_1$ is sparse with sparsity density $p=2n^{-1}$, and with components sampled from the standard normal distribution. Vector $b^1$ is sampled from the uniform distribution in $[0,1]$,  and we enforce $\scal{b^1}{A_1 b^1}<c_1 $, which ensures that $0$ belongs to $\xi_1$. 

Then the remaining ellipsoids, $\xi_2,\ldots,\xi_m$, are constructed in the following form:
\begin{equation}
    \xi_i = \{x\in \re^n\mid \scal{x-x_c^i}{(A_i^\top A_i)^{-1}(x-x_c^i)} \leq 1\},
\end{equation}
where $A_i$ is a positive definite matrix, and $x_c^i\in \re^n$ is the center of $\xi_i$ for $i=2,3,\ldots,m$. To form $\xi_2$, first randomly generate $x_c^2$ of $\xi_2$ outside $\xi_1$. Define $d_2:=\lambda(P_{\xi_1}(c_2)-c_2)$ as the norm of the longest principal semi-axis of $\xi_2$, where $\lambda>1$ is a constant which can decide the intersection is big or small. For ensuring this, we form a diagonal matrix $\Lambda_2=\diag(\|d_2\|,u)$ where $u\in \re^{n-1}$ is a vector whose components are positive and have values less than $\|d_2\|$, and orthogonal matrix $Q_2$ where the first row or column is $\frac{d_2}{\|d_2\|}$. Define $A_2=Q_2 \Lambda_2 Q_2^\top$, and then the ellipsoid $\xi_2$ is complete.

Before forming the remaining ellipsoids, we need to find a fixed point that lies in the intersection of the $m$  ellipsoids. Take point $p = x_c^2 + d_2 \in \xi_1 \cap \xi_2$, and we will guarantee $p\in \xi_i$ for each $i=3,4,\ldots,m$ in the following steps.

% Before constructing the remaining ellipsoids, we need to find a fixed point that will lie in all the ellipsoids. We take $p = x_c^2 + d_2 \in \xi_1 \cap \xi_2$, and we will guarantee that $p\in \xi_i$ for each $i=3,4,\ldots,m$.

Choose an arbitrary point $x_c^i \in \re^n$ which doesn't belong to $\cup_{j=1}^{i-1} \xi_j$, define $d^i=\lambda(p-x_c^i)$ as the norm of the longest semi-axis of $\xi_i$, and generate $A_i$ similarly as was done for $\xi_2$ for each $i=3,\ldots,m$. Repeat this process until we get all $m$ ellipsoids.

% For the comparison, CRM was applied to the equivalent feasibility problem \cref{eq.Product_Space2}, which in formulated in the product space $\re^{nm}$.

The computational experiments were performed on an Intel Xeon W-2133 3.60GHz with 32GB of RAM running Ubuntu 20.04 using  \texttt{Julia  v1.8}~\cite{Bezanson:2017}, and are available at \url{https://github.com/lrsantos11/CRM-CFP}.
The following conditions were used:
\begin{listi}
    \item A random initial point $x^0\in \re^n$ is sampled for the standard normal distribution, ensuring that $x^0\notin \xi_i$, for all $i=1,\ldots,m$; note that for  \texttt{CRMprod}, the initial point is  $(x^0,x^0,\ldots,x^0) \in \re^{nm}$;
    \item   We use the method described in \cite[Alg.~6]{Jia:2017} to compute the projections onto the ellipsoids.  The number of projections onto ellipsoids per iteration that each algorithm requires differs:  \texttt{Alg1} involves $\num{4}$, \texttt{Alg3} asks for  $5$, while both \texttt{SePM} and  \texttt{CRMprod} demand $m$. 
    \item We compare total number of projections until achieve precision and not number of iterations, and also register CPU time (in seconds) for each algorithm.
    \item A limit of \num{30000} total number of projections is enforced.
    \item After each iteration, we calculate the \emph{current error} given by
    \begin{equation}
        e_k=\sum_{i=1}^m \|P_{C_i}(x^k)-x^k\|,
    \end{equation}
    and we establish the \emph{stopping criterion} as
    \begin{equation}
        % \label{eq.Stopping_Criterion}
        \max\{\|x^{k+1}-x^k\| ,  e_k \} \leq \varepsilon,
    \end{equation}
    where $\varepsilon= \num{e-6}$.
    
    \item For each pair $(n,m)$, for $n\in \{20,50,100\}$ and $m\in \{5, 10, 20\}$ we repeat the experiment \num{20} times.  
 
\end{listi}

\Cref{tab.Ellipsoid_iteration,tab.Ellipsoid_Time} summarize the results, in which we exhibit the mean and the standard deviation of CPU running time (in seconds) and total number of projections, respectively, for each algorithm. We also sum up our numerical findings in \Cref{pic.1,pic.2}, by means of the so-called  performance profiles from~\cite{Dolan:2002}. \emph{Performance profiles} allow one to benchmark different methods on a set of problems with respect to a performance measure (in our case, CPU Time and total number of projections). The vertical axis indicates the percentage of problems solved, while the horizontal axis indicates the corresponding factor of the performance index used by the best solver.

We briefly comment on these results. Our numerical findings show that \texttt{Alg1} and  \texttt{Alg3} are faster (in terms of CPU time) than their counterparts (see \Cref{tab.Ellipsoid_Time} and \Cref{pic.1}). \texttt{Alg3}  
is the one with the less number of total projections to achieve the required tolerance, which is expected, as it uses only functional evaluation to determine the control sequence (see \Cref{tab.Ellipsoid_iteration} and \Cref{pic.2}). We remark that \texttt{Alg1} and \texttt{Alg3} perform similarly in terms of CPU time, and when the dimension are higher, the difference between their performance with respect to \texttt{SePM} and \texttt{CRMprod} is more evident.

\begin{table}[htpb]
    \caption{The \texttt{mean} $\pm$ \texttt{std} CPU time (in seconds) comparison per dimension and number of sets.}
\label{tab.Ellipsoid_Time}
\centering 
\begin{tabular}{rr
    S[round-mode = uncertainty, round-precision = 3, table-number-alignment=center,table-format = 1.4+-1.4]
    S[round-mode = uncertainty, round-precision = 3, table-number-alignment=center,table-format = 2.4+-1.4]
    S[round-mode = uncertainty, round-precision = 3, table-number-alignment=center,table-format = 2.4+-1.4]
    S[round-mode = uncertainty, round-precision = 3, table-number-alignment=center,table-format = 2.4+-1.4]}
\toprule 
$n$  & $m$  &   \texttt{Alg1}         & \texttt{Alg3} & \texttt{SePM} & \texttt{CRMprod} \\
\midrule 
20  &  5  &   0.0198124  +-    0.038853   &      0.0368  +-     0.1171   &    0.0333974 +-   0.0584851 &    0.0350442 +-   0.057604 \\
20  & 10  &   0.00982177 +-    0.0102313  &      0.0094 +-     0.0092 &    0.0191617 +-   0.0232425 &    0.0335212 +-   0.035002 \\
20  & 20  &   0.0096 +-    0.0087 &      0.0134539  +-     0.0111134  &    0.0202698 +-   0.0242554 &    0.0595656 +-   0.069896 \\
50  &  5  &   0.0417429  +-    0.0280462  &      0.0497845  +-     0.0587683  &    0.0686484 +-   0.0531841 &    0.0704168 +-   0.050617 \\
50  & 10  &   0.0372108  +-    0.02933    &      0.0421624  +-     0.0412113  &    0.0679019 +-   0.0595938 &    0.0784705 +-   0.075930 \\
50  & 20  &   0.0371539  +-    0.023343   &      0.0411425  +-     0.0233964  &    0.0676161 +-   0.0546071 &    0.212343  +-   0.191711 \\
100  &  5  &   0.383882   +-    0.610646   &      0.353618   +-     0.612096   &    0.637482  +-   0.909212  &    0.532929  +-   0.950897 \\
100  & 10  &   0.359866   +-    0.61953    &      0.360117   +-     0.624696   &    0.623179  +-   0.929969  &    0.776991  +-   0.989221 \\
 100  &   20  & 
 \sisetup{round-mode = uncertainty, round-precision = 4, table-number-alignment=center,table-format = 2.2+-3.2}\num{60.4449     +-  268.584}   &   
 \sisetup{round-mode = uncertainty, round-precision = 4, table-number-alignment=center,table-format = 2.2+-3.2}\num{76.0533     +-   338.5}    &  
 \sisetup{round-mode = uncertainty, round-precision = 4, table-number-alignment=center,table-format = 2.2+-3.2}\num{73.7023    +- 326.609}    &  
 \sisetup{round-mode = uncertainty, round-precision = 4, table-number-alignment=center,table-format = 3.2+-3.2}\num{173.824     +-    772.267}     \\
\bottomrule
\end{tabular}
\end{table}

\begin{figure}[htpb]
    \centering
    {\includegraphics[width=.8\textwidth]{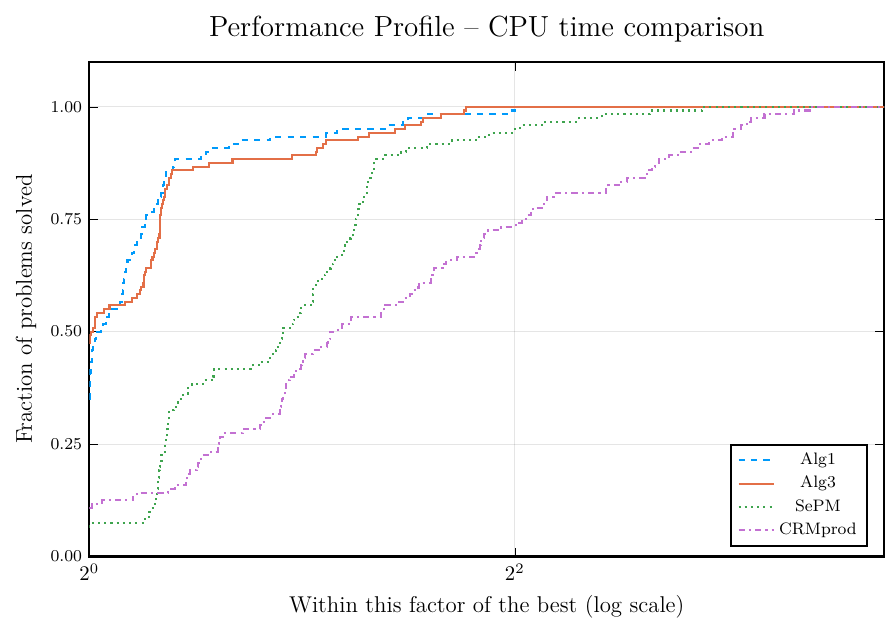}}
    \caption{Performance profile of the experiments considering  CPU Time (in seconds).}
    \label{pic.1}
   \end{figure}

\begin{table}[htpb]
    \caption{The \texttt{mean} $\pm$ \texttt{std} of total projections  per dimension and number of sets.}
\label{tab.Ellipsoid_iteration}
\centering 
\begin{tabular}{rrS[round-mode = uncertainty, round-precision = 3,table-format = 3.2 +- 2.2, table-number-alignment=center]
                  S[round-mode = uncertainty, round-precision = 3,table-format = 2.3 +- 2.2, table-number-alignment=center]
                  S[round-mode = uncertainty, round-precision = 4,table-format = 2.3 +- 3.2, table-number-alignment=center]
                  S[round-mode = uncertainty, round-precision = 4,table-format = 2.3 +- 3.2, table-number-alignment=center]}
\toprule
$n$  & $m$  &   \texttt{Alg1}         & \texttt{Alg3} & \texttt{SePM} & \texttt{CRMprod} \\
   \midrule
   20 &   5  &    51.0 +-   13.7267  &   13.5  +-   4.00657  &   32.0  +-    26.0263  &       66.0 +-     53.178 \\ 
   20 &  10  &   100.0 +-   30.4354  &   13.0  +-   4.4129   &   64.0  +-    65.486   &      183.5 +-    224.013 \\ 
   20 &  20  &   196.0 +-   54.9066  &   13.5  +-   4.00657  &  121.0  +-   120.608   &      588.0 +-    876.966 \\ 
   50 &   5  &    51.0 +-   12.0961  &   14.25 +-   4.37547  &   34.0  +-    25.3709  &       65.5 +-     49.335 \\ 
   50 &  10  &   102.0 +-   27.4533  &   14.0  +-   4.47214  &   66.0  +-    54.9066  &      142.5 +-    186.375 \\ 
   50 &  20  &   208.0 +-   54.4446  &   14.5  +-   4.26121  &  142.0  +-   116.781   &      827.0 +-    1016.34 \\ 
  100 &   5  &    58.0 +-   14.3637  &   15.25 +-   4.12789  &   45.5  +-    26.102   &       66.0 +-     60.166 \\ 
  100 &  10  &   120.0 +-   34.3358  &   15.5  +-   4.26121  &   96.5  +-    64.4225  &      255.5 +-    272.155 \\ 
  100 & 20   & 
  \sisetup{round-mode = uncertainty, round-precision = 5}  \num{764.0  +- 2381.61}  &    
  \sisetup{round-mode = uncertainty, round-precision = 5}  \num{57.75 +- 188.864}   &    
  \sisetup{round-mode = uncertainty, round-precision = 5}  \num{683.0  +- 2176.26}  &  
  \sisetup{round-mode = uncertainty, round-precision = 6}  \num{3005.0  +- 10336.0} \\
\bottomrule
\end{tabular}
\end{table}

\begin{figure}[htpb]
    \centering
{\includegraphics[width=.8\textwidth]{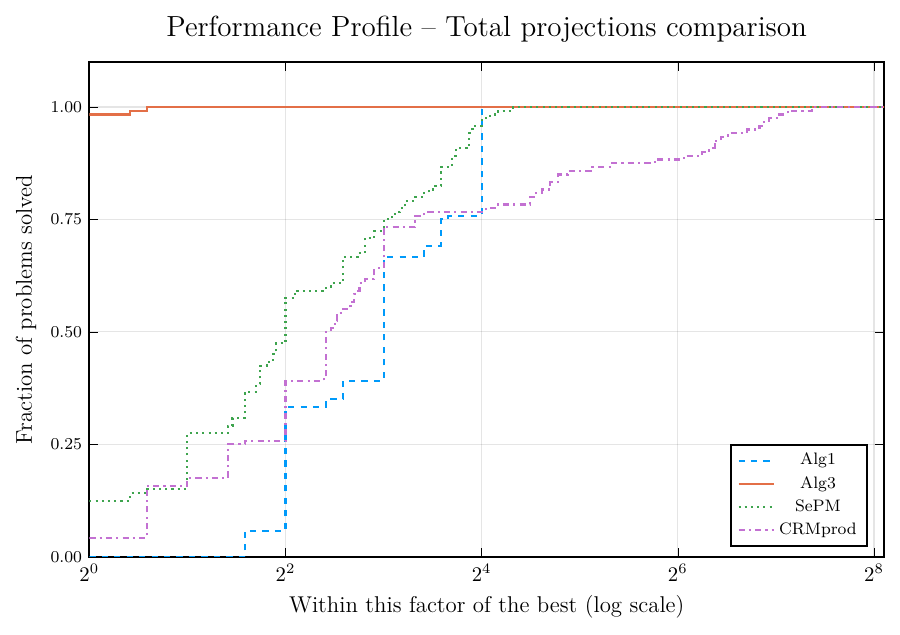}}
\caption{Performance profile of the experiments considering total number of projections.}
 \label{pic.2}
\end{figure}

\section{Concluding remarks}

In this paper we have presented the successive centralized circumcentered-reflection method (s-cCRM) extending cCRM for solving the multiset convex feasibility problem. 
Our theoretical analysis encompasses  the global convergence of s-cCRM and, additionally under an error bound condition, we prove linear convergence of the method. Moreover, we have shown that the s-cCRM is superlinearly convergent under smoothness of the boundaries of the target sets. Furthermore, the numerical experiments illustrate the proposed version of s-cCRM  have better performance than SePM and CRM-Prod. Extensions of novel circumcenter-based iterations by using more natural centralized procedures and possible applications to structured optimization problems are left for future research.

% Data availability

\backmatter

\bmhead{Acknowledgments}

The authors would like to thank the anonymous referees for their valuable comments and suggestions that helped to improve the quality of the paper. The authors also thank the  Brazilian agencies \emph{Conselho Nacional de Desenvolvimento Cient\'ifico e Tecnológico} (CNPq), and \emph{Fundação de Amparao à Pesquisa do Estado do Rio de Janeiro} (FAPERJ), as well as the United States agency \emph{National Science Foundation} (NSF) for their financial support. \textbf{RB} was partially supported by the CNPq Grants 304392/2018-9 and 429915/2018-7, and FAPERJ Grant E-26/201.345/2021; \textbf{YBC} was partially supported by the  NSF Grant DMS-2307328, and by an internal grant from NIU. \textbf{LRS} was partially supported by CNPq Grant 113190/2022-0.

\section*{Declarations}

\bmhead{Data availability} The data and code that support the findings of this study are fully available at \url{https://github.com/lrsantos11/CRM-CFP} or can be obtained from the corresponding author upon request.

\bmhead{Conflict of interest} The authors have no relevant financial or non-financial interests to disclose. 

\bmhead{Ethical statement} We certify that all authors are complying with the journal's ethical policies and that this manuscript has not been published or submitted simultaneously for publication elsewhere.

% % This section is a summary of the major results of the paper, without formulas.

% % \begin{acknowledgements}
% % Acknowledgements to sponsoring agencies and individuals should be placed here.
% % \end{acknowledgements}

% %References
% \bibliographystyle{spmpsci}
\bibliography{refs}

\end{document}